\documentclass[11pt,twoside]{amsart}
\usepackage[latin9]{inputenc}
\usepackage{verbatim}
\usepackage{calc}
\usepackage{amsthm}
\usepackage{amstext}
\usepackage{amssymb}
\usepackage{esint}
\usepackage[unicode=true,pdfusetitle,
 bookmarks=true,bookmarksnumbered=false,bookmarksopen=false,
 breaklinks=false,pdfborder={0 0 1},backref=section,colorlinks=false]
 {hyperref}
\makeatletter

\numberwithin{equation}{section}
\numberwithin{figure}{section}
\theoremstyle{plain}
\newtheorem{thm}{\protect\theoremname}
  \theoremstyle{definition}
  \newtheorem{defn}[thm]{\protect\definitionname}
  \theoremstyle{remark}
  
  \theoremstyle{plain}
  \newtheorem{lem}[thm]{\protect\lemmaname}
  \theoremstyle{plain}
  \newtheorem{prop}[thm]{\protect\propositionname}

\usepackage{amssymb,amsthm,amsmath,amsfonts,amscd}
\usepackage{graphicx}
\usepackage{url}
\usepackage{color}
\usepackage[active]{srcltx}
\usepackage[matrix,arrow]{xy}
\usepackage{mathrsfs}
\usepackage{enumerate}
\usepackage{amsopn} % provides \DeclareMathOperator
\usepackage{bbm} % additional fonts
\usepackage{ulem} %underline etc.
\usepackage{cite}
\usepackage{hyperref}
\allowdisplaybreaks[1]
\usepackage[english]{babel}
\usepackage{bbm}
\usepackage{stmaryrd}
\usepackage{mdef}
\date{\today}
\makeatother
\newcommand {\no}{\noindent}
  \providecommand{\definitionname}{Definition}
  \providecommand{\lemmaname}{Lemma}
  \providecommand{\propositionname}{Proposition}
  \providecommand{\remarkname}{Remark}
\providecommand{\theoremname}{Theorem}

\begin{document}

\title{Scalar conservation laws with multiple rough fluxes}
\begin{abstract}
\noindent We study pathwise entropy solutions  for scalar conservation laws with inhomogeneous fluxes and quasilinear multiplicative rough path dependence. %with spatially dependent fluxes and multiple driving rough paths. 
This extends the previous work  of Lions, Perthame and Souganidis who considered  spatially independent and  inhomogeneous fluxes with multiple paths and a  single driving singular path respectively. The approach is motivated by the theory of stochastic viscosity solutions which relies on special test functions constructed by inverting locally the flow of the stochastic characteristics. For conservation laws this is best implemented at the level of the kinetic formulation which we follow here.
\end{abstract}

\author{Benjamin Gess}

\address{Department of Mathematics \\
University of Chicago \\
Chicago, IL 60637 \\
USA }

\email{gess@uchicago.edu}

\author{Panagiotis E. Souganidis }

\address{Department of Mathematics \\
University of Chicago \\
Chicago, IL 60637 \\
USA }

\email{souganidis@math.uchicago.edu }



\thanks{Benjamin Gess has been partially supported by the research project ``Random dynamical systems and regularization by noise for stochastic partial differential equations'' funded by the German Research Foundation. Panagiotis Souganidis is supported by the NSF grants DMS-0901802 and DMS-1266383.}

\maketitle

\section{Introduction}

\noindent We are interested in the stochastic scalar conservation law (SSCL for short)
\begin{equation}
\begin{cases}
du+ \div_x ({A}(x,u) \circ d{z})=0 \quad\text{ in }\quad\R^{N}\times(0,T),\\[2mm]
%{\displaystyle \sum_{i=1}^N \sum_{j=1}^M } \partial_{x_{i}}(A^{i,j}(u,x))\circ d{\bf z}^{j}=0\quad\text{ in }\quad\R^{N}\times(0,T),}\\[2mm]
u=u_{0}\quad\text{ on }\quad\R^{N}\times\{0\},
\end{cases}\label{eq:scl}
\end{equation}
where  ${A}$ is a smooth $N\times M$ matrix and 
%=(A^{i,j})_{\substack{i=1\dots N \\ j=1\dots M }}$ is a smooth, matrix valued function and 
${z}=(z_1,\ldots, z_M)$ is a geometric $\a$-Hölder rough path; the precise assumptions are given in Section \ref{sec:review_kinetic}. As a particular example  may consider \eqref{eq:scl} with  $A$ a diagonal $N\times N$ matrix  and ${z}$ an $N$-dimensional Brownian motion enhanced to a rough path, in which case \eqref{eq:scl} becomes 
%The precise assumptions are given in Section \ref{sec:review_kinetic}.
\begin{equation}\label{eqn:std_SSCL}
\begin{cases}
du + \sum _{j=1}^N \partial _{x_j} A_j(x,u) \circ dz_j \quad\text{ in }\quad\R^{N}\times(0,T),\\[2mm]
u=u_{0}\quad\text{ on }\quad\R^{N}\times\{0\}.
\end{cases}
\end{equation}
\smallskip

\noindent  There is, of course, a close connection with the theory of stochastic viscosity solutions.  Indeed when $N=1$ 
%in since in the one-dimensional case Stochastic conservation laws like \eqref{eq:scl} arise in various contexts.  There is the connection with Hamilton-Jacobi equations in the one-dimensional case where, 
if $v$ solves the stochastic Hamilton-Jacobi equation %Let us start by reminding the reader of the connection with one-dimensional stochastic Hamilton-Jacobi-Bellman (HJB) equations of the type
%\[
$dv+A(\partial_{x}v,x)\circ dz=0,$
%\]
then, %where $\b$ is a real-valued Brownian motion. Informally 
formally, $u=\partial_{x}v$ satisfies the SSCL %we obtain the following stochastic scalar conservation law of the type \eqref{eq:scl}: 
%\[
$du+\partial_{x}(A(u,x))\circ dz=0.$
%\]
\smallskip

\noindent  SSCL also arise in several applications. A concrete example is provided by  the  theory of mean field games developed by Lasry and Lions (\cite{LL06,LL06-2,LL07}). It turns out that   the mean field limit,  as $L\to\infty$, of the empirical law of the solution  $(X^{1},\dots,X^{L}) \in \R^{N\times L}$ of the stochastic differential equation
\[
dX_{t}^{i}={\bf \s}\Big(X_{t}^{i},\frac{1}{L}\sum_{j\ne i}\d_{X_{t}^{j}}\Big)\circ d{z}_{t}\qquad \text{for} \ i=1,\dots,L,
\]
where $\s$ is a $N\times M$ Lipschitz matrix  on $\R^{N}\times\mcP(\R^{N})$  with adjoint $\s^*$, $\mcP(\R^N)$ is the space of probability measures on $\R^N$  and ${ z}=(z_1,\ldots,z_M)$ is an $N$-dimensional Brownian motion, % It turns out that in the mean field limit ($L\to\infty$) the empirical law of $(X_{t}^{1},\dots,X_{t}^{L})$ 
converges to a measure $\pi_{t}$ which evolves according to %the solution  $m$ to 
\begin{equation}\label{eq:SCL-intro}
  dm +  \text{div}_x [ \s^*(x,m) \circ d{z}]=0;     %\sum_{i=1}^{N}\partial_{x_{i}}(\s^{*}(x,m)m)\circ d\b^{i}. 
\end{equation}
%We note that the $x$-dependency of the flux $A(x,m)=\s^{*}(x,m)m$ is crucial in this application. While such $x$-dependent stochastic scalar conservation laws arise naturally in the applications only simplified cases could be covered in the literature so far. Our results include equations of the type \eqref{eq:SCL-intro} with $\s^*$ 
%\smallskip
%\end{document}
\no note that  \eqref{eq:SCL-intro} is not a ``standard'' conservation law (i.e. of the type \eqref{eqn:std_SSCL}) unless $\s^*$ is a diagonal matrix. For this reason we study here SCCL of the more complex type \eqref{eq:scl}.

\no The study of pathwise solutions to SSCL was initiated by Lions, Perthame and Souganidis (\cite{LPS13}) who introduced the notion of pathwise (stochastic)  entropy solutions for "standard" SSCL with $x$-independent fluxes and multiple continuous paths, that is    
%Scalar conservation laws with stochastic flux have been studied intensively in recent years. In \cite{LPS13} Lions, Perthame and Souganidis introduced the notion of pathwise stochastic entropy solutions, which is based on the kinetic formulation,  and considered  the \textit{$x$-independent} case driven by arbitrary continuous paths,  that is 
%\end{document}
\begin{equation} \label{eq:x-indep}
du+ \sum_{i=1}^{N}\partial_{x_{i}}A_{i}(u)\circ dz_{i}=0,%\label{eq:x-indep}
\end{equation}
for   ${z}=(z_1,\ldots,z_N)$  continuous. %has been first solved by introducing an appropriate concept of stochastic kinetic solutions. 
This work was subsequently extended in \cite{LPS14} to $x$-dependent fluxes driven by a single continuous path  $z$, that is to SSCL of the form 
\begin{equation}
du+{\displaystyle \sum_{i=1}^{N}\partial_{x_{i}}A_{i}(u,x)\circ dz=0}.\label{eq:single-bm}
\end{equation}
%\smallskip

\no The main contribution of this paper  is the extension of the well-posedness theory of the pathwise (stochastic) entropy solutions to general inhomogeneous fluxes  with multi-dimensional driving rough paths and non-standard SSCL of the type \eqref{eq:scl}. %We also obtain%, that is  for \eqref{eq:scl} in its general form. 
\smallskip

\no The theory  is based on the general concepts/philosophy around the theory of stochastic viscosity solutions
for fully nonlinear first- and second-order PDE including stochastic Hamilton-Jacobi equations which was introduced and developed by Lions and Souganidis in \cite{LS98,LS98-2,LS02,LS00-2,LS00}.  The latter is based %theory of stochastic viscosity solutions to \eqref{eq:intro-HJB-general} is based 
on a notion of solution which does not depend on the (blowing up) derivatives and integrals of the driving paths. This is accomplished with the introduction of a class of test functions which are short-time smooth solutions of the stochastic Hamilton-Jacobi part of  the equation (constructed typically by the method of characteristics)  and, when inserted in the equation, 
%who of the paths is to find test functions which do not depend on derivatives of the driving paths, which do not exist, and also when inserted in the equation 
at least formally ``eliminate'' the stochastic part. This is basically a local change of the unknown which however cannot be done unless there is regularity. 
\smallskip

%\no We note the similarity to the theory of stochastic viscosity solutions to stochastic first- and second-order fully nonlinear equations of the form  
%. More precisely, in \cite{LS98,LS98-2,LS02,LS00-2,LS00} the stochastic viscosity solutions approach to fully nonlinear SPDE of the type 
%\begin{equation}\label{eq:intro-HJB-general}
%du+F(D^{2}u,Du,u,x,t)  =\sum_{i=1}^{N}H^{i}(Du,x)\circ d\b^{i},
%\nonumber 
%\end{equation}
%which was introduced and developed by Lions and Souganidis in \cite{LS98,LS98-2,LS02,LS00-2,LS00}. While the notion of stochastic viscosity solutions in principle makes sense in the general case \eqref{eq:intro-HJB-general}, well-posedness could so-far only be shown in either the $x$-independent case or for a single continuous driving noise or 
%($\b$ being real-valued), leaving the case of $x$-dependent Hamiltonian $H$ and multiple driving signals $\b$ as a long-standing open problem. 
%for multiple brownian paths \cite{LS14}.
%\smallskip

\no The notion of pathwise (stochastic) entropy solutions builds on the above. It is based on the kinetic formulation of conservation laws and test functions which are propagating along the characteristics of the corresponding linear transport equation.  The latter are defined globally in time and, in this aspect, the test functions are easier to construct since it is not necessary to invert characteristics.  There  are, however, new difficulties since the test functions are not easily localizable. 
%and rough kinetic solutions to \eqref{eq:scl} are closely related. Both notions rely on setting up a robust notion of a solution (i.e. getting rid of the occurence of derivatives of the driving signal, to the extend possible) by testing the SPDE against test-functions transported along the corresponding stochastic characteristics. As a result, a good understanding of these characteristics is at the heart of the proofs of well-posedness, thus explaining the difficulties in the general $x$-dependent case. 
%\smallskip
%\no The notion of solutions to \eqref{eq:scl}, \eqref{eq:x-indep} and \eqref{eq:single-bm} is similar in the sense that in each case one first passes to a kinetic formulation which is then tested against test functions transported along the corresponding stochastic characteristics (see section \ref{sec:review_kinetic} below). In principle this leads to a robust notion of solution to \eqref{eq:scl} which  does involve the (blowing up) derivatives of the driving signal.  However, proving well-posedness of the solutions 
% is of different level of difficulty among these three classes. 
For $x$-independent fluxes the characteristics can be solved explicitly and, hence, it is possible to keep track of the cancellations that are taking place.  Such explicit solutions are not available  when there is spatial dependence. In the presence of one driving force the characteristics can be expressed, after a change of time, in terms of the ones for a conservation law without the rough signals, and this assists in taking care of the cancellations. When dealing with multiple driving paths such expressions are not possible and it becomes necessary to use more sophisticated tools from the rough path theory.  In this note we establish the uniqueness and existence of pathwise solutions. For the latter,  following   \cite{P02}, we prove the existence of what we call generalized pathwise entropy solutions, which are easier to construct using weak limits in the kinetic equations, and then show that they are indeed pathwise solutions. To obtain the weak limits we prove some estimates about the defect measure which are new even in the ``deterministic'' setting. 
%such solutions are indeed we prove there existence and stability and then show that they actually give rise to the pathwise solutions. 
It is worth remarking that proving strong compactness requires BV-type estimates which are, however, not at all obvious due to the inhomogeneity of the equations and the singularities of the paths. Such estimates were obtained in \cite{LPS13} but  there the fluxes were $x$-independent.  %We prove there existence and stability and then show that they actually give rise to the pathwise solutions. 
 %This makes tracking the cancellations possible. the key reason being the more explicit formulae for the stochastic characteristics being available for \eqref{eq:x-indep} and \eqref{eq:single-bm} (for more details see below).
\smallskip

\no As already outlined above, the concept of pathwise rough entropy/kinetic  solutions   builds upon the 
%The approach to \eqref{eq:scl} developed in this paper is based on a kinetic formulation of \eqref{eq:scl}. A stable notion of a pathwise entropy/kinetic solution to \eqref{eq:scl} is then derived using the 
characteristic system corresponding to the kinetic formulation. Due to the roughness of the driving signal ${z}$ the this system becomes a stochastic differential equation (SDE). In order to justify this notion of a solution and to construct solutions we will consider approximating problems based on smoothed signals $z^n \in C^\infty([0,T];\R^M)$. An essential ingredient of the construction will thus be the continuous dependence of the characteristics on the driving signal ${z}$. However, it is well-known that in general SDE do not depend continuously on the paths of the driving signal.  %with respect to $C^0([0,T];\R^M)$-topology. 
This is where rough paths theory enters the picture. The relevance of rough paths in this work is that they provide %will be the definition of 
stronger topologies on the space of the driving signals making the solution to SDE, in our case the system of characteristics, continuous with respect to the driving signal.
\smallskip

\no We remark that the term ``stochastic'' may not be the best choice to characterize  the solutions we study which rely on rough paths and use very little if at all stochasticity. A better name, which we use here,  is pathwise rough entropy/kinetic solutions. To avoid any confusion we emphasize that the solutions we study here coincide with the ones  in \cite{LPS13, LPS14} when considering  the setting of these references.
\smallskip

\noindent Recently Debussche and Vovelle \cite{DV10},  Feng and Nualart \cite{FN08} and Chen, Ding and Karlsen \cite{CDK12} (see also Debussche, Hofmanova and Vovelle \cite{DHV13} and Hofmanova \cite{H13-5}) have put forward a theory of weak entropy solutions of scalar conservation laws with It\^o-type semilinear and not quasilinear stochastic dependence. Such problems do not appear amenable to a pathwise theory. Our results do not cover the equations studied in \cite{DV10, FN08} and vice versa. We refer to Section 6 of \cite{LPS13} for a discussion of these issues. A rough paths approach to SSCL with linear transport noise has been developed in \cite{FG14}.
\smallskip

 \noindent We conclude the introduction  emphasing that this paper builds upon the ideas of   \cite{LPS13, LPS14} and, of course,  the general methodology of the work on stochastic viscosity solutions \cite{LS98,LS98-2,LS02,LS00-2,LS00}.  We expect that several extensions may be possible
without a real change in the general strategy. For example, it should take only few and straightforward technical modifications to extend  the results to equations like \eqref{eq:scl} plus  a semilinear term.   We leave such extensions to the interested reader.

%that is prove the well posedness of pathwise entropt/rough kinetic solutions to}  

%\end{document}
%\textcolor{blue}{
%\begin{equation}
%\begin{cases}
%du+ \div_x ({A}(x,u) \circ d{z})= g(x,u) \circ dy  \quad\text{ in }\quad\R^{N}\times(0,T),\\[2mm]
%{\displaystyle \sum_{i=1}^N \sum_{j=1}^M } \partial_{x_{i}}(A^{i,j}(u,x))\circ d{\bf z}^{j}=0\quad\text{ in }\quad\R^{N}\times(0,T),}\\[2mm]
%u=u_{0}\quad\text{ on }\quad\R^{N}\times\{0\},
%\end{cases}\label{eq:scl100}
%\end{equation}}
%\textcolor{blue}{ 
%\noindent  \textcolor{blue}{with $A$ and $z$ are as before and $y=(y^1,\ldots, y^N)$ is another rough signal and $g:\R^{N+1} \to \R$ satisfies appropriate %regularity assumptions. We leave the details to the interested reader.}

\subsection*{Organization of the paper.}
The paper is organized as follows: In Section~\ref{sec:review_kinetic} we state the assumptions, we review briefly some facts about entropy solutions and their kinetic formulation  and we introduce the notion of pathwise rough entropy/kinetic  solutions. Their uniqueness is  proved in Section~\ref{sec:unique}; a key step in the proof is a technical lemma about solutions of differential equations with rough paths which is presented in the Appendix.  The construction of pathwise rough entropy/kinetic  solutions, which 
requires %is based on the concept of generalized pathwise entropy solutions and 
some some new estimates, %introduced in Section~\ref{sec:generalized_basic}, will be 
is presented in Section~\ref{sec:existence}. %The proof of existence of (generalized) solutions is based on some new estimates presented in Section~\ref{sec:stable_apriori}. 
In the two Appendices  we recall some facts  from the theory of rough paths used in the main body of the paper and prove an estimate that is used throughout the paper.
%respectively to the proofs of the contraction  and intrinsic uniqueness of the stochastic entropy solutions and their existence. 
%Some final observations are listed in section~\ref{sec:fo}.

\subsection*{Notation and terminology} We work in the $N$-dimensional Euclidean space $\R^N.$  If $x\in \R^N$ and $\xi \in \R$,  $\left\Vert 
\begin{array}{cc}
x\\
\xi
\end{array}\right\Vert $ denotes the norm of $(x,\xi)$ in $\R^{N+1}$. 
 If ${b}=(b^1,\ldots,b^N):\R^N \to \R^N$ is a smooth vector field, $\text{div} b:=\Sigma_{i=1}^N \partial_{x_i} b^i.$ We say that $A \subset [0,T]$ is null  if it has zero Lebesgue measure.  The space of bounded measures in $\R^{N}\times \R$ is denoted by $\mcM(\R^{N}\times \R).$ For a set $C \subset [0,T]$ we define
 $\D(C):=\{(s,t):\ s,t\in C,\ s\le t\}.$
 %For $\a\in(0,1)$, $C^{0,\a}([0,T]; G^{\lfloor\frac{1}{\a}\rfloor}(\R^N))$ is the space of $\a$-Holder maps with values  in $G^{\lfloor\frac{1}{\a}\rfloor}(\R^N)$, the set XXX, $\| \cdot \|_{\a-\Hoel;[0,T]} $is the Holder seminorm on $[0,T]$,
For $\d>0$, $\text{Lip}^\d(\R^{N}\times\R;\R^{N})$ is the set of functions with $k=0, \ldots\lfloor\delta\rfloor$ bounded derivatives and $\d-\lfloor\delta\rfloor$ Hölder continuous $\lfloor\delta\rfloor$-th derivative. If $\Delta =\{a=t_0 \leq t_i\leq \ldots \leq t_M=b\}$ is a partition of $[a,b]\subset [0,T]$,  then $\|\Delta\| :=\{|t_{i+1}-t_1|: i=0,\ldots,M-1\}.$ %For a matrix-valued function $A \in C^1(\R^N;\R^{N\times M})$, $(\div A)^j := \sum_{i=1}^N \sum_{j=1}^M \partial_{x_i} A^{ij}$ where, for $i=1,\dots,N$ and $j=1,\dots,M$,   $A^{ij}$ is the $(i,j)$ element of $A$. If  $A\in C^{1}(\R^{N}\times\R; \R^{N\times M}),$ then the $(i,j)$-element $a^{ij}$ of the matrix ${a}(x,\xi)=:A_\xi(x, \xi) \in \R^{N\times M}$ is $A_\xi^{ij}(x, \xi)$, while, for each $j=1,\ldots, M$, the $j$-component  $b(x, \xi) =: \div A(x, \xi) \in \R^M$
%is $\sum_{1=1}^N \partial_{x_i} A^{ij}(x, \xi).$

%$ {\bf a}(x,\xi)=A_u(x, \xi):= (A_u^{i,j}(x, \xi))_{\substack{i=1\dots N\\j=1\dots M}} $ and ${\bf b}(x, \xi) = \div A(x, \xi) = (\div A^j(x, \xi))_{j=1\dots M}$

%\end{document}

\section{The kinetic formulation and pathwise rough %stochastic 
entropy/kinetic solutions}\label{sec:review_kinetic}

%\framebox{\begin{minipage}[t]{1\columnwidth}%
%Would need kinetic theory for $(x,t)$-depending fluxes. Ok for entropy solutions, for kinetic solutions no reference known, but should be ok. Maybe an appendix defining the kinetic solution and proving equivalence to entropy solutions in this general case. %
%\end{minipage}}

\subsection*{Assumptions} We present the assumptions we need to make to study \eqref{eq:scl}. %and , assuming $A\in C^{1}(\R^{N}\times\R),$  we introduce the notation 
%\begin{equation}
%\begin{cases}
%{\bf a}(x,\xi)=A_u(x, \xi):= (A_u^{i,j}(x, \xi))_{\substack{i=1\dots N\\j=1\dots M}}\\[2mm]   
%   {\bf b}(x, \xi) = \div A(x, \xi) = (\div A^j(x, \xi))_{j=1\dots M}%(b_1(x,\xi),\ldots,b_N(x,\xi)):=(A_{x_1}^1(x, \xi), \ldots, A_{x_N}^N(x, \xi))
 %\end{cases} 
%\label{ab}
%\end{equation}

%\no We now summarize the main assumptions we need in the paper.  

\no The first is that, for some $\a\in(0,1),$
\begin{equation}\label{z}
 z  \text{ is  an $\a$-Hölder geometric rough path,}
 \end{equation}
 that is %\mathbf{z} 
 $z \in C^{0,\a}([0,T];G^{\lfloor\frac{1}{\a}\rfloor}(\R^{N})).$ For some background on rough paths and the meaning of $G^{\lfloor\frac{1}{\a}\rfloor}(\R^{N})$ we refer to the Appendix \ref{app:RP}. %for some $\a\in(0,1)$. 
 \smallskip
 
 \no As as far as the flux is concerned we assume  that, for some $M\in \N$,  %$\gamma >\frac{1}{\a}\ge1$,
\begin{equation}\label{ab} 
 A \in C^1(\R^N \times \R; \R^{N\times M}).
 \end{equation}
 
% \no To state the rest of the assumptions we recall that, for  $A \in C^1(\R^N;\R^{N\times M})$, 
% $$(\div A)^j := \sum_{i=1}^N \sum_{j=1}^M \partial_{x_i} A^{ij},$$ 
%where, for $i=1,\dots,N$ and $j=1,\dots,M$,   $A^{ij}$ is the $(ij)$ element of $A$. %If  $A\in C^{1}(\R^{N}\times\R; \R^{N\times M}),$ and 

\no  We consider the matrix $a \in \R^{N \times M}$ and vector  $b \in \R^M$ given by 
$${a}(x,\xi)=:A_\xi(x, \xi) \in \R^{N\times M} \  \text{and} \ b(x, \xi) =: \div A(x, \xi);$$
note that the $(ij)$-element $a^{ij}$ of the matrix ${a}(x,\xi)$ %=:A_\xi(x, \xi) \in \R^{N\times M}$
is $A_\xi^{ij}(x, \xi)$ and $(\div A)^j := \sum_{i=1}^N  \partial_{x_i} A^{ij},$ where, for $i=1,\dots,N$ and $j=1,\dots,M$,   $A^{ij}$ is the $(ij)$ element of $A$. 
% while, for each $j=1,\ldots, M$, the $j$-component  $b(x, \xi) =: \div A(x, \xi) \in \R^M$
%is $\sum_{1=1}^N \partial_{x_i} A^{ij}(x, \xi).$

%$ {\bf a}(x,\xi)=A_u(x, \xi):= (A_u^{i,j}(x, \xi))_{\substack{i=1\dots N\\j=1\dots M}} $ and ${\bf b}(x, \xi) = \div A(x, \xi) = (\div A^j(x, \xi))_{j=1\dots M}$

%and%, for some $\gamma >\frac{1}{\a}\ge1$, 
%\begin{equation}\label{flux}
%{\bf A} \in C(\R^N \times \R; \R^N),
%\end{equation} 
%\begin{equation}
%{\bf b}(x,0) = 0,
%\label{as:b}
%\eeq
%and
%\beq\label{assumptions1}
%\p_j a_i, \;  \p_u a_i,  \;  \p_j b,  \;  \p_u b \in L^\infty (\R^N \times \R) \  \text{for} \ 1\leq i,j  \leq N.
%\eeq
%\noindent The theory we are developing deals with solutions in $(L^1\cap L^\infty)(\R^N)$. As a result it is not necessary to introduce additional 
%requirements on the flux, which would be needed if were dealing with only $L^1$-solutions.
\smallskip

\no We assume that, for some $\gamma >\frac{1}{\a}\ge1$,
\begin{equation}
a,b\in\Lip^{\g+2}(\R^{N}\times\R),\label{eq:regularity_assumptions}
\end{equation}
and, for all $x\in\R^{N},$
%for some $\gamma>\frac{1}{\a}\ge1$. Here $\Lip^{\gamma}$ is to be understood in the sense of E. Stein, i.e. $k$ bounded derivatives for $k=0,\dots,\lfloor\gamma\%rfloor$ and $\gamma-\lfloor\gamma\rfloor$ Hölder continuous $\lfloor\gamma\rfloor$-th derivative. In addition we assume
\begin{equation}
{b}(x,0)=0. %\quad\forall x\in\R^{N}.
\label{eq:b_ass}
\end{equation}
%where $a$ and $b$ where introduced at the end of the notation subsection of the introduction.
\smallskip

\subsection*{Kinetic formulation}  The well established theories of entropy solutions and kinetic solutions (see, for example, Dalibard \cite{D06} and Perthame \cite{P02}) extend easily to problems with smooth driving paths, that is  for $z\in C^{1}([0,T];\R^{M}).$  With this assertion at hand, below we recall the basic facts. 
%$there are  well-established theories of entropy solutions and kinetic solutions may be applied to \eqref{eq:scl} (cf. \cite{D06,P02}). 
%In order to fix notations we recall some of the basic ideas next. 
\smallskip
\no Given the nonlinear function%Given the nonlinear function
%\subsection*{The kinetic formulation} Given the nonlinear function
{\begin{equation}
\chi(x,\xi,t):=\chi(u(x,t),\xi):=\left\{ \begin{array}{l}
+1 \ \text{ if } \ 0\leq\xi\leq u(x,t),\\[2mm]
-1 \ \text{ if } \ u(x,t)\leq\xi\leq0,\\[2mm]
\;0 \ \text{ otherwise},
\end{array}\right.\label{eq:char_fctn}
\end{equation}
we may rewrite \eqref{eq:scl} in its kinetic form. Recalling that here we are assuming smooth driving signals, to simplify the notation we  write
%\begin{align*}
%\begin{equation*}
$$a_{i}(x,\xi,t) :=\sum_{j=1}^M(\partial_{u}A^{i,j})(x,\xi)\dot{z}^{j}(t),$$
and
$$b(x,\xi,t) :=\sum_{i=1}^N\sum_{j=1}^M (\partial_{x_{i}}A^{i,j})(x,\xi)\dot{z}^{j}(t)=\div A(x,\xi)\cdot \dot{z}(t).$$
%\end{align*}
In view  of  \eqref{eq:b_ass} we have
\[
b(x,0,t)={b}(x,0)\cdot \dot{z}(t)=0.
\]
Fix $T>0$.  The kinetic form of \eqref{eq:scl} is  %(cf. \cite{D06}) 
\begin{equation}\label{eq:kinetic_form-1}%{align}
\begin{cases}
\partial_t \chi+ a(x,\xi,t)\cdot D_{x}\chi-b(x,\xi,t)\partial_{\xi}\chi  = \partial_{\xi} m \ \text{ in } \ \R^{N}\times\R\times(0,T),\\
%\\%[1.5mm]
\chi  =\chi(u_{0}(\cdot),\cdot) \ \text{ on } \ \R^{N}\times\R\times\{0\},
\end{cases}
\end{equation}
where
\begin{equation}\label{measure0} 
m \  \text{is a bounded  nonnegative measure on }  \R^{N}\times\R\times[0,T]; 
\end{equation}
the precise bounds on the mass of the measure are stated later.%TODO

\no Due to the Hamiltonian structure of \eqref{eq:kinetic_form-1}, we may rewrite it as
\begin{equation}
\begin{cases}
\partial_t 
\chi+\div_{x}({\displaystyle a(x,\xi,t)\chi)-\partial_{\xi}(b(x,\xi,t)\chi)}  =\partial_{\xi}m\text{ in }\R^{N}\times\R\times(0,T),\\[1mm]
\label{eq:kinetic_form-1-1}
\chi  =\chi(u_{0}(\cdot),\cdot)\text{ on }\R^{N}\times\R\times\{0\}.\nonumber 
\end{cases}
\end{equation}

\subsection*{Derivation of a stable notion of kinetic solutions}
As stated above the notion of kinetic solutions  is not well defined for rough driving signals, since the coefficients $a,b$ blow up with $\dot{z}$. On the other hand,  following \cite{LPS13, LPS14}, we observe that the linearity of \eqref{eq:kinetic_form-1}  in $\chi$ suggests that we may use the characteristics of \eqref{eq:kinetic_form-1}
to derive a stable notion of solution, 
which we will call pathwise rough entropy/kinetic solutions. %stochastic 
%entropy solutions. 
% Let
% \begin{align*}
% a_{i}(x,\xi) & :=(\partial_{u}A^{i})(x,\xi)\\
% b(x,\xi) & :=\sum_{i}(\partial_{x_{i}}A)(x,\xi).
% \end{align*}
%\end{document}
\smallskip

\no For now  we continue assuming  that  $z\in C^1([0,T];\R^M)$ and, for every $t_{0}\in [0,T]$, we consider the corresponding forward/backward transport equation
\begin{equation}\label{eq:transport}
\begin{cases}
\partial_t \vr_{t_{0}}+{\displaystyle a(x,\xi,t)\cdot D_{x}\vr_{t_{0}}-b(x,\xi,t)\partial_{\xi}\vr_{t_{0}}}  =0 \ \text{ in } \ \R^{N}\times\R\times \R,\\[1mm]%\label{eq:transport}%\\
\vr_{t_{0}}  =\vr^{0} \ \text{ on } \ \R^{N}\times\R\times\{t_{0}\},
\end{cases}
\end{equation}
which is rewritten as
 \begin{equation}\label{eq:transport_1}
\begin{cases}
\partial_t \vr_{t_{0}}+\sum_{i=1}^N \sum_{j=1}^M (\partial_{u}A^{i,j})(x,\xi)\partial_{x_{i}}\vr_{t_{0}}\dot{z}^{j}(t)-\sum_{i=1}^N \sum_{j=1}^M(\partial_{x_{i}}A^{i,j})(x,\xi)\partial_{\xi}\vr_{t_{0}}\dot{z}^{j}(t)  =0  \\%\ \text{ in } \ \R^{N}\times\R\times \R,\\[1mm]%\label{eq:transport_1}%
\vr_{t_{0}}(t_{0})  =\vr^0.  %\ \text{ on } \ \R^{N}\times\R\times \{t_{0} \};%\nonumber 
\end{cases}
\end{equation}
We note that in view of the linearity of \eqref{eq:transport}  and \eqref{eq:transport_1}, it makes sense to consider solutions starting at $t_0$ and existing forward and backward in time. 
 \smallskip
  
\no For each $(y,\eta)\in\R^{N+1}$, $t_0\ge 0$ and $\vr^0 \in C_{c}^{\infty}(\R^{N}\times\R)$ we consider the solution $\vr_{t_{0}}=\vr_{t_{0}}(x,y,\xi,\eta,t)$ to \eqref{eq:transport} with  $\vr_{t_{0}}(\cdot,y,\cdot,\eta,t_0)=\vr^0(\cdot-y,\cdot-\eta)$. %at time $t=t_{0}$. 
Defining the convolution along characteristics by
\[
\vr_{t_{0}}\ast\chi(y,\eta,t):=\int\vr_{t_{0}}(x,y,\xi,\eta,t)\chi(x,\xi,t)dxd\xi,
\]
we find (see  Lemma \ref{lem:kinetic-pathwise} below) 
\[
\partial_{t}\vr_{t_{0}}\ast\chi(y,\eta,t)=-\int\partial_{\xi}\vr_{t_{0}}(x,y,\xi,\eta,t)m(x,\xi,t)dxd\xi\  \ \text{in }\R^{N}\times\R\times(0,T).
\]

\subsection*{The characteristics} It is a classical fact that the solution of \eqref{eq:transport} can be expressed in terms of the associated (backward) characteristics starting at $t_0 \geq 0$. %which we describe next.  
Note that,  in contrast to \eqref{eq:transport},  the characteristic equations are well-defined also for rough driving signals $z$. 

\no We assume next that ${z}$ is an $\a$-Hölder geometric rough path and we consider, for $i=1,\dots,N$,  the rough differential equations
\begin{equation}
\begin{cases}\label{eq:rough-char}
d{Y}_{(t_{0},y,\eta)}^{i}(t) = \sum_{j=1}^M a^{i,j}(Y_{(t_{0},y,\eta)}(t),\z_{(t_{0},y,\eta)}(t))d{ z}^j(t), \quad  Y_{(t_{0},y,\eta)}^{i}(t_{0}) =y^{i} ,\\[1.5mm]
d{\z}_{(t_{0},y,\eta)}(t)  =-\sum_{i=1}^N \sum_{j=1}^M (\partial_{x_{i}}A^{i,j})(Y_{(t_{0},y,\eta)}(t),\z_{(t_{0},y,\eta)}(t))d{ z}^{j}(t), \quad 
 \z_{(t_{0},y,\eta)}(t_{0})  =\eta.%\nonumber
\end{cases}
\end{equation}
%with initial condition
%\begin{align*}
%Y_{(t_{0},y,\eta)}^{i}(t_{0}) & =y^{i}\\
%\z_{(t_{0},y,\eta)}(t_{0}) & =\eta
%\end{align*}
%for $i=1,\dots,N$. 
We note that, in light of \eqref{eq:regularity_assumptions} and \cite{FV10}, there exits  a unique solution to \eqref{eq:rough-char} and, for $Y_{(t_{0},y,\eta)}(t):=(Y_{(t_{0},y,\eta)}^{1}(t), \ldots, Y_{(t_{0},y,\eta)}^{N}(t))$,  the map  
\[
\left(\begin{array}{cc}
Y_{t_{0}}\\
\z_{t_{0}}
\end{array}\right)(t):\left(\begin{array}{cc}
y\\
\eta
\end{array}\right)\mapsto\left(\begin{array}{cc}
Y_{(t_{0},y,\eta)}(t)\\
\z_{(t_{0},y,\eta)}(t)
\end{array}\right),
\]
is a flow of homeomorphisms on $\R^{N+1}.$

\no For $t\ge t_0$  set
\[
\left(\begin{array}{cc}
Y_{(t,y,\eta)}(t_{0})\\
\z_{(t,y,\eta)}(t_{0})
\end{array}\right):=\left[\left(\begin{array}{cc}
Y_{t_{0}}\\
\z_{t_{0}}
\end{array}\right)(t)\right]^{-1}\left(\begin{array}{cc}
y\\
\eta
\end{array}\right)%\quad\text{for }t\ge t_{0}.
\]

\no and observe that, in view of \eqref{eq:b_ass}, 
\[
\z_{(t_{0},y,0)}\equiv 0,% \quad\forall t\in[0,T].
\]
and, hence, for all $t$,
\begin{equation}\label{eq:sign_char}
\sgn(\z_{(t_{0},y,\eta)}(t))=\sgn(\eta). %\quad\forall t\in[0,T].\label{eq:sign_char}
\end{equation}
\no For each time $t_{1}\ge0$ and for  $i=1,\dots,N$ we  consider the backward characteristics
\begin{align*}
d{X}_{(t_{1},x,\xi)}^{i}(t) & = \sum_{j=1}^M a^{i,j}(X_{(t_{1},x,\xi)}(t),\Xi_{(t_{1},x,\xi)}(t))d{z}^{t_{1},j}(t),\\
d{\Xi}_{(t_{1},x,\xi)}(t) & =-\sum_{i=1}^N \sum_{j=1}^M(\partial_{x_{i}}A^{i,j})(X_{(t_{1},x,\xi)}(t),\Xi_{(t_{1},x,\xi)}(t))d{z}^{t_{1},j}(t),\\
X_{(t_{1},x,\xi)}^{i}(0) & =x^{i} \  \text{and} \ 
\Xi_{(t_{1},x,\xi)}(0) =\xi.
\end{align*}
%with $i=1,\dots,N$ and 
%$$
%X_{(t_{1},x,\xi)}^{i}(0)  =x^{i} \ \  \text{and}  \ \
%\Xi_{(t_{1},x,\xi)}(0) =\xi,
%$$
%\begin{align*}
%X_{(t_{1},x,\xi)}^{i}(0) & =x^{i}\\
%\Xi_{(t_{1},x,\xi)}(0) & =\xi,
%\end{align*}
where, for $t\in[0,t_1]$,   ${z}^{t_{1}}$ is the time-reversed rough path, that is  %TODO
$$
{z}^{t_{1}}(t):={z}(t_{1}-t). % ,%\quad\text{for }t\in[0,t_1].
$$
Then, for all $t\in[t_{0},t_{1}],$  we have
%\begin{align*}
$$X_{(t_{1},Y_{(t_{0},y,\eta)}(t_{1}),\z_{(t_{0},y,\eta)}(t_{1}))}(t_{1}-t)  =Y_{(t_{0},y,\eta)}(t)$$
and
$$\Xi_{(t_{1},Y_{(t_{0},y,\eta)}(t_{1}),\z_{(t_{0},y,\eta)}(t_{1}))}(t_{1}-t)  =\Xi_{(t_{0},y,\eta)}(t).$$
%\end{align*}
%for all $t\in[t_{0},t_{1}]$. 
%\end{document}
In particular, the inverse of the homeomorphism  $\left(\begin{array}{cc}
Y_{t_{0}}\\
\z_{t_{0}}
\end{array}\right)(t_{1})$ is given by 
\[
\left(\begin{array}{cc}
Y_{(t_1,x,\xi)}(t_{0})\\
\z_{(t_1,x,\xi)}(t_{0})
\end{array}\right)=\left[\left(\begin{array}{cc}
Y_{t_{0}}\\
\z_{t_{0}}
\end{array}\right)(t_{1})\right]^{-1}\left(\begin{array}{cc}
x\\
\xi
\end{array}\right)=\left(\begin{array}{cc}
X_{(t_{1},x,\xi)}(t_{1}-t_{0})\\
\Xi_{(t_{1},x,\xi)}(t_{1}-t_{0})
\end{array}\right).
\]
Hence, the solution $\vr_{t_{0}}$ to \eqref{eq:transport} with $\vr_{t_{0}}(\cdot,y,\cdot,\eta,t_0)=\vr^0(\cdot-y,\cdot-\eta)$ is given,  for all $t\in[0,T]$, by
\begin{equation}
\vr_{t_{0}}(x,y,\xi,\eta,t)
=\vr^{0}\left(\begin{array}{cc}
X_{(t,x,\xi)}(t-t_{0})-y\\
\Xi_{(t,x,\xi)}(t-t_{0})-\eta
\end{array}\right)
=\vr^{0}\left(\begin{array}{cc}
Y_{(t,x,\xi)}(t_{0})-y\\
\z_{(t,x,\xi)}(t_{0})-\eta 
\end{array}\right).%,\quad\forall t\in[0,T].
\label{eq:transport_stable}
\end{equation}
\smallskip

\subsection*{Pathwise rough entropy/kinetic solutions}  We introduce the definition of the pathwise rough entropy/kinetic solution.
%/rough kinetic solution --- recall the definition of  $\chi$ in \eqref{eq:char_fctn}

\begin{defn}
\label{def:path_e-soln}Let $u_{0}\in(L^{1}\cap L^2)(\R^{N})$. A function $u\in L^{\infty}([0,T];L^{1}(\R^{N}))$ is a pathwise rough entropy/kinetic  solution to 
%rough kinetic solution 
 \eqref{eq:scl}, if there exists a nonnegative bounded measure $m$ on $\R^{N}\times\R\times[0,T]$ such that, for all $t_{0}\ge0$, all test functions $\vr_{t_{0}}$ given by \eqref{eq:transport_stable} with $\vr^{0}\in C_{c}^{\infty}(\R^{N+1})$ and $\vp\in C_{c}^{\infty}([0,T)),$ %we have
%\begin{align*}
 $$ \int_{0}^{T}\partial_{t}\vp(r)(\vr_{t_{0}}\ast\chi)(y,\eta,r)dr+\vp(0)(\vr_{t_{0}}\ast\chi)(y,\eta,0)\\
  =\int_{0}^{T}\int\vp(r)\partial_{\xi}\vr_{t_{0}}(x,y,\xi,\eta,r)m(x,\xi,t)dxd\xi dr.$$
%\end{align*}
%for all $\vp\in C_{c}^{\infty}([0,T))$. 
%where $\chi$ is as in \eqref{eq:char_fctn}.
\end{defn}

\smallskip
\no Throughout the paper we will use a different almost pointwise in form of the definition. Since this  will be used several times in the paper, we state it as a separate Proposition.

\begin{prop}\label{rmk:rough_kinetic} Let $u_{0},m$ be as in Definition \ref{def:path_e-soln}. Then $u\in L^{\infty}([0,T];L^{1}(\R^{N}))$ is a pathwise rough entropy/kinetic solution to \eqref{eq:scl} if and only if there exists a  null set $\mcN\subseteq[0,T]$  such that  $0\not\in\mcN$ and, for all $t_{0}\ge0$, $(s,t)\in \D\left([0,T]\setminus\mcN\right)$ and $(y,\eta) \in \R^{N+1}$ and all test functions $\vr_{t_{0}}$ given by \eqref{eq:transport_stable} with $\vr^{0}\in C_c^{\infty}(\R^{N+1})$, 
\begin{align}
\vr_{t_{0}}\ast\chi(y,\eta,t) - \vr_{t_{0}}\ast\chi(y,\eta,s)=& -\int_{s}^{t}\int\partial_{\xi}\vr_{t_{0}}(x,y,\xi,\eta,r)m(x,\xi,r)dxd\xi dr.\label{eq:gen_kinetic_integrated-1}
%(\vr_{t_{0}}\ast\chi)(y,\eta,\cdot)|_{s}^{t}= & -\int_{s}^{t}\int\partial_{\xi}\vr_{t_{0}}(x,y,\xi,\eta,r)m(x,\xi,r)dxd\xi dr.\label{eq:gen_kinetic_integrated-1}
\end{align}
\end{prop}

\no Note that applying Lebesgue's differentiation theorem both for right-handed and centered averages yields  that the entropy defect measure $m$ does not have point masses on $[0,T]\setminus\mcN$. In particular, the integrals appearing in \eqref{eq:gen_kinetic_integrated-1} are well-defined. 
%\end{rem}
\begin{proof}
An application of Lebesque's differentiation theorem implies the claim  for a null  $\mcN\subseteq [0,T]$ possibly depending on $t_0$ and $\vr^0$. We then choose countable  dense subset $A\subseteq [0,T]$ and a countable $B\subseteq C_c^\infty(\R^{N+1})$ % where $C_c^\infty(\R^{N+1})$ is endowed 
such that,  for each $\vr^0 \in  C_c^\infty(\R^{N+1})$, there exists  a sequence $\vr^{0,n} \in B$ with uniformly compact support, such that $\vr^{0,n} \to \vr^0$ in $C^1(\R^{N+1}).$ %topology.
\smallskip
   
\no   Since $A,B$ are countable, there exists a null $\mcN\subseteq[0,T]$ such that \eqref{eq:gen_kinetic_integrated-1} holds for all $t_0\in A$, $\vr^0\in B$. For arbitrary $t_0\in [0,T]$, $\vr^0 \in  C_c^\infty(\R^{N+1})$ we may then choose approximating sequences  $t_0^n \in A$, $\vr^{0,n} \in B$ with $t_0^n\to t$, $\vr^{0,n}\to\rho$ in $C^1$%_{loc}(\R^{N+1})$ topology 
and $\vr^{0,n}$ having uniformly compact support. Then \eqref{eq:gen_kinetic_integrated-1} is satisfied for all $t_0^n$,  $\vr^{0,n}$, that is 
   \begin{align}
   \vr^n_{t_{0}^n}\ast\chi(y,\eta,t) - \vr^n_{t_{0}^n}\ast\chi(y,\eta,s)=& -\int_{s}^{t}\int\partial_{\xi}\vr^n_{t_{0}^n}(x,y,\xi,\eta,r)m(x,\xi,r)dxd\xi dr.\label{eq:gen_kinetic_integrated-2}
   \end{align}
   
\no   In view of  \eqref{eq:transport_stable} and the continuity in time of the characteristics $(X,\Xi)$ (see Appendix \ref{app:RP})  we may take the limit $n\to \infty$ in the left hand side of \eqref{eq:gen_kinetic_integrated-2}. As for the right hand side goes we note that 
   \begin{align*}
	\partial_{\xi}\vr_{t_{0}^n}^n(x,y,\xi,\eta,t) 
	&= D_x\vr^{0,n}\left(\begin{array}{cc}
     X_{(t,x,\xi)}(t-t_{0}^n)-y\\
     \Xi_{(t,x,\xi)}(t-t_{0}^n)-\eta
     \end{array}\right)\partial_\xi X_{(t,x,\xi)}(t-t_{0}^n)\\
     &+\partial_\xi\vr^{0,n}\left(\begin{array}{cc}
          X_{(t,x,\xi)}(t-t_{0}^n)-y\\
          \Xi_{(t,x,\xi)}(t-t_{0}^n)-\eta
          \end{array}\right)\partial_\xi \Xi_{(t,x,\xi)}(t-t_{0}^n).
   \end{align*}
  The continuity of the characteristics %(and rough paths estimates cf. Appendix \ref{app:RP})  
   yields
     $$	\partial_{\xi}\vr_{t_{0}^n}^n(x,y,\xi,\eta,t) \to 	\partial_{\xi}\vr_{t_{0}}(x,y,\xi,\eta,t)  \ \text{ in } \  C^1(\R^{N+1})$$
%   in $C^1_{loc}(\R^{N+1})$. 
Since the $\vr^{0,n}$'s  have uniformly compact support, we may thus take the limit $n\to \infty$ in the right hand side of \eqref{eq:gen_kinetic_integrated-2} to conclude.
\end{proof}

\no Next we show that, for smooth paths, the notions of kinetic and pathwise rough entropy/kinetic solutions are equivalent.

\begin{lem}
\label{lem:kinetic-pathwise}Assume that $z\in C^{1}([0,T];\R^{N})$ and  $u\in L^{\infty}([0,T];L^{1}(\R^{N})).$  Then $u$ is a pathwise rough entropy/kinetic solution to \eqref{eq:scl} if and only if  $u$ is a kinetic solution.
% to \eqref{eq:scl}.
\end{lem}
\begin{proof}
Let $u$ be a kinetic solution to \eqref{eq:scl}. Then,  for all $\vp\in C_{c}^{1}(\R^{N}\times\R\times [0,T)),$
$$\int_{0}^{T}\int\chi\left(\partial_t \vp+a(x,\xi,r)\cdot D_{x}\vp-b(x,\xi,r)\partial_{\xi}\vp\right)d\xi dxdr+\int\chi(x,\xi,0)\vp(x,\xi,0)d\xi dx\\
 =\int_{0}^{T}\int m\partial_{\xi}\vp d\xi dxdr.$$

%\begin{align*}
 %& \int_{0}^{T}\int\chi\left(\partial_t \vp+a(x,\xi,r)\cdot D_{x}\vp-b(x,\xi,r)\partial_{\xi}\vp\right)d\xi dxdr+\int\chi(x,\xi,0)\vp(x,\xi,0)d\xi dx\\
 %& =\int_{0}^{T}\int m\partial_{\xi}\vp d\xi dxdr.
%\end{align*}
%for all $\vp\in C_{c}^{1}(\R^{N}\times\R\times [0,T))$. 
\no It follows that  there exists a null set $\mcN\subseteq[0,T]$ such that  $0\not\in\mcN$ and, for all  $(s,t)\in \D\left([0,T]\setminus\mcN\right)$ and  all $\vp\in C_{c}^{1}(\R^{N}\times\R\times [0,T]),$
\begin{align*}
&\int_{s}^{t}\int\chi (x,\xi,r) \left(\partial_t\vp(x,\xi,r) + a(x,\xi,r)\cdot D_{x}\vp(x,\xi,r)-b(x,\xi,r)\partial_{\xi}\vp(x,\xi,r)\right)d\xi dxdr + \\  
& \int\chi(x,\xi,t)\vp(x,\xi,t)d\xi dx - 
 \int\chi(x,\xi,s)\vp(x,\xi,s)d\xi dx =\int_{s}^{t}\int m\partial_{\xi}\vp d\xi dxdr.
\end{align*}
%\begin{align*}
% & \int_{s}^{t}\int\chi\left(\partial_t\vp+a(x,\xi,r)\cdot D_{x}\vp-b(x,\xi,r)\partial_{\xi}\vp\right)d\xi dxdr+\int\chi(x,\xi,\cdot)\vp(x,\xi,\cdot)d\xi dx|_{s}^{t}\\
% & =\int_{s}^{t}\int m\partial_{\xi}\vp d\xi dxdr.
%\end{align*}
%for all $\vp\in C_{c}^{1}(\R^{N}\times\R\times [0,T])$. 

\no Choosing $\vp=\vr_{t_{0}}$ yields
\begin{align*}
 & \int\chi(x,\xi,t)\vr_{t_{0}}(x,y,\xi,\eta,t)d\xi dx -  \int\chi(x,\xi,s)\vr_{t_{0}}(x,y,\xi,\eta,s)d\xi dx =\int_{s}^{t}\int m\partial_{\xi}\vr_{t_{0}}d\xi dxdr,
\end{align*}
which by Remark \ref{rmk:rough_kinetic} implies that $\chi$ is a pathwise rough entropy/kinetic solution to \eqref{eq:scl}.
\smallskip

\no In Theorem \ref{thm:unique} below we prove the uniqueness of pathwise rough entropy/kinetic solutions. Combined with the existence of kinetic solutions (see  \cite{D06,P02}) this implies that every pathwise rough entropy/kinetic  solution is a ``classical'' kinetic solution.
\end{proof}

% \subsection{Stable solution to \eqref{eq:transport}}

\section{The uniqueness of pathwise rough entropy/kinetic 
solutions}\label{sec:unique}
\no We discuss here the following theorem that yields  the uniqueness of pathwise rough entropy/kinetic solutions.

\begin{thm}\label{thm:unique}
Let $u^{(1)},u^{(2)}  \in L^{\infty}([0,T];L^{1}(\R^{N}))$ be two pathwise rough entropy/kinetic  solutions to \eqref{eq:scl} with initial values $u_{0}^{1},u_{0}^{2}\in(L^{1}\cap L^{2})(\R^{N})$. Then, for a.e. $t\in [0,T],$
\begin{equation}\label{eq:gen_kinetic_unique}
  \|u^{(1)}(t)-u^{(2)}(t)\|_{L^1(\R^d)} \le \|u^{1}_0-u^{2}_0\|_{L^1(\R^d)}.
\end{equation}
%\begin{align}char_
%\int(\chi^{(1)}(y,\eta,t)-\chi^{(2)}(y,\eta,t))^{2}dyd\eta & \le\int(\chi^{(2)}(y,\eta,0)-\chi^{(2)}(y,\eta,0))^{2}dyd\eta\label{eq:gen_kinetic_unique}\\
% & =\int|u_{0}^{1}(y)-u_{0}^{2}(y)|dy.\nonumber 
%\end{align}
\end{thm}

\no Before we present the details, we describe briefly the main idea of the proof  and motivate some of the technicalities. The key observation, that goes back to \cite{P02},  is that 
\begin{align*}
 \|u^{(1)}(t)-u^{(2)}(t)\|_{L^1(\R^d)}&=\int (\chi^{(1)}(x,\xi,t) -\chi^{(2)}(x,\xi,t))^2dx d\xi\\ 
&= \int |\chi^{(1)}(x,\xi,t)| + |\chi^{(2)}(x,\xi,t)|  -2\chi^{(1)}(x,\xi,t) \chi^{(2)}(x,\xi,t) dx d\xi,
\end{align*}
where, for $i=1,2$, $\chi^{(i)}$ is related to $u^{(i)}$ by \eqref{eq:char_fctn}. 
\no To obtain the contraction estimate \eqref{eq:gen_kinetic_unique}, it then suffices, in principle, to show that  the derivative with respect to $t$ of the right hand side of the above equality is nonpositive; actually  this is how uniqueness is  shown for conservation laws with smooth time dependence. The difficulty here is that it is not possible to perform this differentiation due to the lack of regularity. Following instead \cite{LPS13, LPS14} we replace $\chi^{(1)}, \chi^{(2)}$ by 
$\vr_{t_0} \ast \chi^{(1)}$,  $\vr_{t_0} \ast  \chi^{(2)}$ for suitable choices of $\vr_{t_0}$. Then it is possible to differentiate with respect to $t$ at the expense of creating  several additional terms that need to be estimated. 
\smallskip

%\no We continue with the proof of Theorem~\ref{thm:unique}.
\begin{proof}[Proof of Theorem \ref{thm:unique}]
For a fixed $t_{0} \in [0,T]$ we have 
\begin{align*}
 & -2\frac{d}{dt}\int\left(\chi^{(1)}\ast\vr_{t_{0}}\right)\left(\chi^{(2)}\ast\vr_{t_{0}}\right)dyd\eta\\
 & =-2\int\left(\chi^{(1)}\ast\vr_{t_{0}}\right)\frac{d}{dt}\left(\chi^{(2)}\ast\vr_{t_{0}}\right)+\left(\chi^{(2)}\ast\vr_{t_{0}}\right)\frac{d}{dt}\left(\chi^{(1)}\ast\vr_{t_{0}}\right)dyd\eta\\
 & =-2\int\int \chi^{(1)}(x',\xi',t)\vr_{t_{0}}(x',y,\xi',\eta,t)\vr_{t_{0}}(x,y,\xi,\eta,t)\partial_{\xi}m^{(2)}(x,\xi,t)dxd\xi dx'd\xi'dyd\eta\\
 & -2\int\int m^{(1)}(x',\xi',t)\vr_{t_{0}}(x',y,\xi',\eta,t)\vr_{t_{0}}(x,y,\xi,\eta,t)\partial_{\xi}m^{(1)}(x,\xi,t)dxd\xi dx'd\xi'dyd\eta\\
 & =-2\int\int \chi^{(1)}(x',\xi',t)\partial_{\xi'}\vr_{t_{0}}(x',y,\xi',\eta,t)\vr_{t_{0}}(x,y,\xi,\eta,t)m^{(2)}(x,\xi,t)dxd\xi dx'd\xi'dyd\eta\\
 & -2\int\int \chi^{(2)}(x',\xi',t)\partial_{\xi'}\vr_{t_{0}}(x',y,\xi',\eta,t)\vr_{t_{0}}(x,y,\xi,\eta,t)m^{(1)}(x,\xi,t)dxd\xi dx'd\xi'dyd\eta\\
 & +Err^{(1,2)}(t_{0},t)\\
 & =2\int\int\left(\d(\xi')-\d(\xi'-u^{(1)}(x',t))\right)\vr_{t_{0}}(x',y,\xi',\eta,t)\vr_{t_{0}}(x,y,\xi,\eta,t)m^{(2)}(x,\xi,t)dxd\xi dx'd\xi'dyd\eta\\
 & +2\int\int\left(\d(\xi')-\d(\xi'-u^{(2)}(x',t))\right)\vr_{t_{0}}(x',y,\xi',\eta,t)\vr_{t_{0}}(x,y,\xi,\eta,t)m^{(1)}(x,\xi,t)dxd\xi dx'd\xi'dyd\eta\\
 & +Err^{(1,2)}(t_{0},t)\\
 & \le2\int\int\vr_{t_{0}}(x',y,0,\eta,t)\vr_{t_{0}}(x,y,\xi,\eta,t)(m^{(2)}+m^{(1)})(x,\xi,t)dxd\xi dx'dyd\eta\\
 & +Err^{(1,2)}(t_{0},t),
\end{align*}
where, the inequality is due to the nonnegativity of the Dirac masses $\delta$ and  
\begin{equation}\label{eq:err-1-1}
%\begin{align}
\begin{cases}
Err^{(1,2)}(t_{0},t):=  2\int\left(\chi^{(1)}(x',\xi',t)m^{(2)}(x,\xi,t)+\chi^{(2)}(x',\xi',t)m^{(1)}(x,\xi,t)\right)\\[1.5mm]
 (\vr_{t_{0}}(x',y,\xi',\eta,t)\partial_{\xi}\vr_{t_{0}}(x,y,\xi,\eta,t)%\\%\label{eq:err-1-1}\\
  +\partial_{\xi'}\vr_{t_{0}}(x',y,\xi',\eta,t)\vr_{t_{0}}(x,y,\xi,\eta,t))dyd\eta dx'd\xi'dxd\xi.
\end{cases}
\end{equation}
%\end{document}
%&Err^{(1,2)}(t_{0},t)=  2\int\left(\chi^{(1)}(x',\xi',t)m^{(2)}(x,\xi,t)+\chi^{(2)}(x',\xi',t)m^{(1)}(x,\xi,t)\right)\label{eq:err-1-1}\\%\nonumber \\
% & (\vr_{t_{0}}(x',y,\xi',\eta,t)\partial_{\xi}\vr_{t_{0}}(x,y,\xi,\eta,t)\\%\label{eq:err-1-1}\\
%  +\partial_{\xi'}\vr_{t_{0}}(x',y,\xi',\eta,t)\vr_{t_{0}}(x,y,\xi,\eta,t))dyd\eta dx'd\xi'dxd\xi.\nonumber 
%\end{align}%TODO
Lemma \ref{lem:abs_val_part} below applied to $f(x,\xi,t):=\chi^{(i)}(x,\xi,t)$ and $m:=m^{(i)}$ yields a null $\mcN\subseteq[0,T]$ with $0\not\in\mcN$ such that, for every $t_{0}\ge0$, all  $(s,t)\in \D\left([0,T]\setminus\mcN\right)$ and $i=1,2$, %we have
\begin{align*}
 & 2\int_{s}^{t}\int\int\vr_{t_{0}}(x',y,0,\eta,r)\vr_{t_{0}}(x,y,\xi,\eta,r)m^{(i)}(x,\xi,r)dxd\xi dx'dyd\eta dr\\
 & =-\int(\sgn\ast\vr^{0})(y,\eta)[(\chi^{(i)}\ast\vr_{t_{0}})(y,\eta,t)- (\chi^{(i)}\ast\vr_{t_{0}})(y,\eta,s)]dyd\eta +\int_{s}^{t}Err^{(i)}(t_{0},r)dr,%\\ 
 %&+\int_{s}^{t}Err^{(i)}(t_{0},r)dr,
\end{align*}
%\begin{align*}
% & 2\int_{s}^{t}\int\int\vr_{t_{0}}(x',y,0,\eta,r)\vr_{t_{0}}(x,y,\xi,\eta,r)m^{(i)}(x,\xi,r)dxd\xi dx'd\xi'dyd\eta dr\\
% & =-\int(\sgn\ast\vr^{0})(y,\eta)(f^{(i)}\ast\vr_{t_{0}})(y,\eta,\cdot)|_{s}^{t}dyd\eta+\int_{s}^{t}Err^{i}(t_{0},r)dr,
%\end{align*}
\no with %$i=1,2$ and
\begin{align*}
Err^{(i)}(t_{0},t)=\int m^{(i)}(x,\xi,t)[ & \int (\partial_{\xi'}\vr_{t_{0}}(x',y,\xi',\eta,t)\vr_{t_{0}}(x,y,\xi,\eta,t) \\
& +\vr_{t_{0}}(x',y,\xi',\eta,t)\partial_{\xi}\vr_{t_{0}}(x,y,\xi,\eta,t))dyd\eta dx'd\xi']dxd\xi.
\end{align*}
It follows that, for all $t_{0}\ge0$ and all $(s,t)\in \D\left([0,T]\setminus\mcN\right)$,
\begin{align}
&-2\int\left[(\chi^{(1)}\ast\vr_{t_{0}})(y,\eta,t)(\chi^{(2)}\ast\vr_{t_{0}})(y,\eta,t) - 
%\int
(\chi^{(1)}\ast\vr_{t_{0}})(y,\eta,s)(\chi^{(2)}\ast\vr_{t_{0}})(y,\eta,s)\right]dyd\eta  \nonumber \\ %|_{s}^{t} 
\le & -\int(\sgn\ast\vr^{0})(y,\eta)[(\chi^{(1)}\ast\vr_{t_{0}})(y,\eta,t) - (\chi^{(1)}\ast\vr_{t_{0}})(y,\eta,s)]dyd\eta \nonumber \\ %|_{s}^{t}dyd\eta
& -\int(\sgn\ast\vr^{0})(y,\eta)[(\chi^{(2)}\ast\vr_{t_{0}})(y,\eta,t) - (\chi^{(2)}\ast\vr_{t_{0}})(y,\eta,s)]dyd\eta\label{eq:unique_1}\\   %|_{s}^{t}dyd\eta\label{eq:unique_1}
& +\int_{s}^{t}Err^{(1)}(t_0, r)+Err^{(2)}(t_0,r)+Err^{{1,2}}(t_0, r)dr. \nonumber 
\end{align}
%\end{document}
%\begin{align}
%-2\int\left(\chi^{(1)}\ast\vr_{t_{0}}\right)\left(\chi^{(2)}\ast\vr_{t_{0}}\right)dyd\eta|_{s}^{t} & \le-\int(\sgn\ast\vr^{0})(y,\eta)(\chi^{(1)}\ast\vr_{t_{0}})(y,\eta,\cdot)|_{s}^{t}dyd\eta\nonumber \\
% & -\int(\sgn\ast\vr^{0})(y,\eta)(\chi^{(2)}\ast\vr_{t_{0}})(y,\eta,\cdot)|_{s}^{t}dyd\eta\label{eq:unique_1}\\
 %& +\int_{s}^{t}Err^{1}(r)+Err^{2}(r)+Err^{1,2}(r)dr,\nonumber 
%\end{align}
%for all $t_{0}\ge0$, $s\le t$, $s,t\in[0,T]\setminus\mcN$. 

\no For $t\in[0,T]$ we define
\begin{align*}
F(t) & :=\int(\chi^{(1)}(y,\eta,t)-\chi^{(2)}(y,\eta,t))^{2}dyd\eta.
\end{align*}
\no Let $\vr_{\ve}^{s},\vr_{\ve}^{v}$ be standard smooth approximations of Dirac masses (here we use the superscripts $s$ and $v$ to signify that they are approximations to Dirac masses in space and velocity respectively)  and, for each $t_{0}\ge0$, let $\vr_{t_{0},\ve}(x,y,\xi,\eta,t)$ be as in \eqref{eq:transport_stable} with initial condition $\vr_{\ve}^{0}(x,y,\xi,\eta)=\vr_{\ve}^{s}(x-y)\vr_{\ve}^{v}(\xi-\eta)$, that is, for all $t\in[0,T],$ 
%\begin{align}
\begin{equation}\label{eq:eps_test_fct}
\vr_{t_{0},\ve}(x,y,\xi,\eta,t)  =\vr_{\ve}^{0}\left(\begin{array}{cc}
X_{(t,x,\xi)}(t-t_{0})-y\\
\Xi_{(t,x,\xi)}(t-t_{0})-\eta
\end{array}\right)
  =\vr_{\ve}^{s}(X_{(t,x,\xi)}(t-t_{0})-y)\vr_{\ve}^{v}(\Xi_{(t,x,\xi)}(t-t_{0})-\eta). %\nonumber 
\end{equation}
We then define
%With $\vr_{t_{0},\ve}$ as in \eqref{eq:eps_test_fct} we define
\begin{align*}
F_{t_{0},\ve}(t):= & -2\int\left(\chi^{(1)}\ast\vr_{t_{0},\ve}\right)(y,\eta,t)\left(\chi^{(2)}\ast\vr_{t_{0},\ve}\right)(y,\eta,t)dyd\eta\\
 & +\int\sgn^{\ve}(\eta)(\chi^{(1)}\ast\vr_{t_{0},\ve})(y,\eta,t)dyd\eta +\int\sgn^{\ve}(\eta)(\chi^{(2)}\ast\vr_{t_{0},\ve})(y,\eta,t)dyd\eta.
% & +\int\sgn^{\ve}(\eta)(\chi^{(2)}\ast\vr_{t_{0},\ve})(y,\eta,t)dyd\eta.
\end{align*}%TODO
Lemma \ref{lem:convergence} below, applied to $f(x,\xi,t):=\chi^{(i)}(x,\xi,t)$, implies that, as $\ve\to 0$ and for all $t_{0}\ge0$ and all $t\in[t_{0},T]\setminus\mcN$,
\begin{align*}
F_{t_{0},\ve}(t)\to & -2\int \chi^{(1)}(y,\eta,t)\chi^{(2)}(y,\eta,t)dyd\eta +\int|\chi^{(1)}|(y,\eta,t)dyd\eta+\int|\chi^{(2)}|(y,\eta,t)dyd\eta,
% & +\int|\chi^{(1)}|(y,\eta,t)dyd\eta+\int|\chi^{(2)}|(y,\eta,t)dyd\eta,
\end{align*}
%for $\ve\to0$ for all $t_{0}\ge0$ and all $t\in[t_{0},T]\setminus\mcN$. 
and thus
%\no Using  that $f^{(i)}\in\{-1,0,1\}$ a.e. (see, for example, Proposition \ref{prop:reconstruction}) we conclude that, as $\ve\to0$ for all $t_{0}\ge0$ and all $t\in[t_{0},T]\setminus\mcN$,
\begin{equation}
F_{t_{0},\ve}(t)\to F(t).\label{eq:F-conv}
\end{equation}
%for $\ve\to0$ for all $t_{0}\ge0$ and all $t\in[t_{0},T]\setminus\mcN$.
\no In view  \eqref{eq:unique_1}, for all $t_{0}\ge0$ and  $(s,t)\in \D\left([0,T]\setminus\mcN\right)$, we find  
\begin{align}
 & F_{t_{0},\ve}(t)-F_{t_{0},\ve}(s)\le\int_{s}^{t}(Err^{(1)}(t_{0},r,\ve)+Err^{(2)}(t_{0},r,\ve)+Err^{(1,2)}(t_{0},r,\ve))dr,\label{eq:first_est}
\end{align}
%for all $t_{0}\ge0$, $s\le t$, $s,t\in[0,T]\setminus\mcN$. Here,
%\end{document}
where
\begin{align*}
& Err^{(1,2)}(t_{0},r,\ve):= 2\int [(\chi^{(1)}(x',\xi',t)m^{(2)}(x,\xi,t)+\chi^{(2)}(x',\xi',t)m^{(1)}(x,\xi,t))\\[1mm]
 & (\vr_{t_{0},\ve}(x',y,\xi',\eta,t)\partial_{\xi}\vr_{t_{0},\ve}(x,y,\xi,\eta,t)
 +\partial_{\xi'}\vr_{t_{0},\ve}(x',y,\xi',\eta,t)\vr_{t_{0},\ve}(x,y,\xi,\eta,t))]dyd\eta dx'd\xi'dxd\xi,
\end{align*}

\no and, for $i=1,2$,
\begin{align*}
Err^{(i)}(t_{0},r,\ve):=\int m^{(i)}(x,\xi,t)( &\int  (\partial_{\xi'}\vr_{t_{0},\ve}(x',y,\xi',\eta,t)\vr_{t_{0},\ve}(x,y,\xi,\eta,t) \\
& +\vr_{t_{0},\ve}(x',y,\xi',\eta,t)\partial_{\xi}\vr_{t_{0},\ve}(x,y,\xi,\eta,t))dyd\eta dx'd\xi')dxd\xi.
\end{align*}

%and $Err^{(1)},Err^{(2)}$ are defined analogously. %TODO

\no Lemma \ref{lem:err_est}  implies that, for all $(s,t)\in\D\left([t_{0},T]\right)$ and $i=1,2$, 
\begin{align*}
& \int_{s}^{t}Err^{(1,2)}(t_{0},r,\ve)dr\le  2\left(\int_{s}^{t}\int m^{(2)}(x,\xi,r)+m^{(1)}(x,\xi,r)dxd\xi dr\right)\\
 & \sup_{\begin{subarray}{c}
(x,\xi) \in \R^{N+1}\\
r\in[s,t]
\end{subarray}}\left(\int\left|\int(\vr_{t_{0},\ve}(x',y,\xi',\eta,r)\partial_{\xi}\vr_{t_{0},\ve}(x,y,\xi,\eta,r)
  +\partial_{\xi'}\vr_{t_{0},\ve}(x',y,\xi',\eta,r)\vr_{t_{0},\ve}(x,y,\xi,\eta,r))dyd\eta\right|dx'd\xi'\right)  \\
& \le C\left(\int_{s}^{t}\int m^{(2)}(x,\xi,r)+m^{(1)}(x,\xi,r)dxd\xi dr\right)(t-t_{0})^{\a},
\end{align*}
and, similarly
\begin{align*}
\int_{s}^{t}Err^{i}(t_{0},r,\ve)\le & C\left(\int_{s}^{t}\int m^{(i)}(x,\xi,r)dxd\xi dr\right)(t-t_{0})^{\a}.
\end{align*}
%for all $s,t\in[t_{0},T]$, $s\le t$, $i=1,2$. 

\no It follows that, for all $(s,t)\in \D\left([t_{0},T]\right),$
\begin{equation}\label{eq:err_est}
  \int_{s}^{t}Err^{1}(t_{0},r,\ve)+Err^{2}(t_{0},r,\ve)+Err^{1,2}(t_{0},r,\ve)dr
  \le C(t-t_{0})^{\a} \int_{s}^{t}\int (m^{(2)}(x,\xi,r)+m^{(1)}(x,\xi,r))dxd\xi dr.\nonumber 
\end{equation}
%for all $s,t\in[t_{0},T]$, $s\le t$. 

\no  Next recall that, in view of  \eqref{eq:F-conv}, \eqref{eq:first_est} and \eqref{eq:err_est},  there exists  a null $\mcN\subseteq[0,T]$ with $0\not\in\mcN$ such that for all $(s,t)\in \D\left([0,T]\setminus\mcN\right)$ and  every partition $\D=\{s=t_{0}\le\dots\le t_{M}=t\}\subseteq[0,T]\setminus\mcN,$  %we have
\begin{align*}
F(t)-F(s) & =\sum_{i=1}^{N}F(t_{i+1})-F(t_{i})\\
 & =\lim_{\ve\to0}\sum_{i=1}^{N}F_{t_{i},\ve}(t_{i+1})-F_{t_{i},\ve}(t_{i})\\
 & \le\lim_{\ve\to0}\sum_{i=1}^{N}\int_{t_{i}}^{t_{i+1}}[Err^{(1)}(t_{i},r,\ve)+Err^{(2)}(t_{i},r,\ve)+Err^{(1,2)}(t_{i},r,\ve)]dr\\
 & \le C\sum_{i=1}^{N}(t_{i+1}-t_{i})^{\a}\left(\int_{t_{i}}^{t_{i+1}}\int [m^{(2)}(x,\xi,r)+m^{(1)}(x,\xi,r)]dxd\xi dr\right)\\
 & \le C\|\D\|^{\a}\left(\int_{s}^{t}\int [m^{(2)}(x,\xi,r)+m^{(1)}(x,\xi,r)]dxd\xi dr\right).
\end{align*}
%where $\|\D\|:=\max \{|t_{i+1}-t_{i}|: i=0,\dots,M-1 \}.$ 

\no Letting $\|\D\|\to0$ yields that, for all $(s,t)\in \D\left([0,T]\setminus\mcN\right)$,
$F(t)\le F(s)$ and, since, 
%for all $s,t\in[0,T]\setminus\mcN$ with $s\le t$. 
$0\not\in\mcN$,  the claim follows.\end{proof}

\section{The existence of pathwise rough entropy/kinetic solutions}\label{sec:existence}
 
\no We establish here the existence of pathwise rough entropy/kinetic solutions, which, in view of Theorem~\ref{thm:unique}, are unique.

\begin{thm}\label{thm:existence}
Let  $u_{0}\in(L^{1}\cap L^{2})(\R^{N})$. Then there exists a pathwise rough entropy/kinetic solution $u\in  L^{\infty}([0,T];L^{1}(\R^{N}))$ to \eqref{eq:scl}.\end{thm} 

\no Since the proof is long we first outline the main steps. The solution is found as a weak limit of solutions of problems like \eqref{eq:scl} with regularized driving signals. Due to the lack of BV-type estimates to obtain the necessary compactness we follow the program of \cite{P02} and work with what are known  as generalized kinetic solutions. %of the approximate problems. 
We then show that these solutions converge weakly to a limit that we call a generalized pathwise rough entropy/kinetic solution.  The final step is to show that the latter is indeed a  pathwise rough entropy/kinetic solution to \eqref{eq:scl}. Each of the steps described above are discussed in a separate subsection.

%We provide next the details.
\smallskip

%he proof of Theorem \ref{thm:existence} proceeds in several steps which we outline next.  %rough kinetic solution. 
\subsection*{The regularized equation and some difficulties}  Let $z^n$ be a regular approximation of  the rough driving signal $z$ (such $z^n$ exist since $z$ is assumed to the a geometric rough path -- see Appendix \ref{app:RP}) and consider the approximating kinetic solution $\chi^n$ to  
%The existence  ofkinetic  solutions $\chi^{n}$ to 
\begin{equation}\label{eq:kinetic_form-1-4-1}
\begin{cases}
\partial_t %\begin{array}{l}
\chi^{n} + a^{n}(x,\xi,t)\cdot D_{x}\chi^{n}-b^{n}(x,\xi,t)\partial_{\xi}\chi^{n} =\partial _\xi m^{n} \ \text{ in } \ \R^{N}\times\R\times(0,T),\\[1mm]%\label{eq:kinetic_form-1-4-1}\\
\chi^{n}  =\chi(u_{0}(\cdot),\cdot) \ \text{ on } \ \R^{N}\times\R\times\{0\},%\nonumber 
\end{cases}
\end{equation}
where
%\begin{align*}
$$ a_{i}^{n}(x,\xi,t)  := \sum_{j=1}^M(\partial_{u}A^{i,j})(x,\xi)\dot{z}^{n,j}(t) \  \  \text{and} \  \
b^{n}(x,\xi,t)  :=\sum_{i=1}^N\sum_{j=1}^M(\partial_{x_{i}}A^{i,j})(x,\xi)\dot{z}^{n,j}(t).$$
%\end{align*}
Although not written down the existence of a unique solution $\chi^{n}$  to \eqref{eq:kinetic_form-1-4-1} follows from straightforward modifications of standard methods (see, for example, \cite{D06,P02}), since $\dot{z}^{n}$ is continuous (actually smooth). In order to pass to the limit $n\to\infty$ we need to establish some uniform in $n$  estimates on $\chi^{n}$ and $m^{n}$ that do not blow up as $n\to\infty$. Unfortunately the known estimates do not have this property. For example, the bound on the total mass of the kinetic measure that can be obtained following \cite{D06,P02}  is 
%The standard  will be derived in the following section. We note that the standard a-priori estimates for $m^{n}$ obtained in \cite{Dalibard} are not useful here because they depend on the total variation of $z^{n}\to\mathbf{z}$, which blows up as $n\to \infty.$  Indeed following \cite{Dalbard}
%at the $n$ level one obtains the following estimate one the mass of $m^n$:
\begin{equation}\label{takis2}
\int_0^T \int_{\R^{N+1}} m^n(x,\xi,t) dx d\xi dt \leq \frac{1}{2}\|u^0\|_2^2 + \int_0^T \int_{\R^{N+1}} |b^n(x,\xi,t)|dx d\xi dt,
\end{equation} 
and the right hand side blows up as $n\to \infty$ since it depends on the total variation of ${ z}^n$. 
Following \cite{LPS14} It is therefore necessary to find  new bounds on the total  mass of $m^n$ which is presented in Section \ref{sec:stable_apriori}. 
\smallskip

\no Once uniform estimates on $\chi^{n}$ and $m^{n}$ have been obtained, we find subsequences, which we denote below the same way as the sequences, which converge in the appropriate weak $\star$ limits in \eqref{eq:kinetic_form-1-4-1}, that is 
%we extract weakly convergent subsequences, which we parametrize again by $n$, 
\begin{equation*}
  \chi^{n} \rightharpoonup f \  \text{in} \ \ L^\infty \  \text{weak} \star \    \text{and} \    m^{n}  \rightharpoonup m \   \text{in} \  \mcM \ \text{weak} \star.
  \end{equation*}
%with the convergence holding in $L^\infty-\text{weak} \ \star$ and in the space of measures respectively.
\no Although a priori it is unclear whether $f$ is again a  characteristic-type function, that is  whether $f(x,\xi,t)=\chi(u(t,x),\xi)$ for some $u$,  we may pass to the limit in \eqref{eq:kinetic_form-1-4-1} to obtain
\begin{equation*}
\begin{cases}
\partial_t %\begin{array}{l}
f + a(x,\xi,t)\cdot D_{x}f-b(x,\xi,t)\partial_{\xi}f =\partial _\xi m \ \text{ in } \ \R^{N}\times\R\times(0,T),\\[1.5mm]%\label{eq:kinetic_form-1-4-1}\\
f  =\chi(u_{0}(\cdot),\cdot) \ \text{ on } \ \R^{N}\times\R\times\{0\}.%\nonumber 
\end{cases}
\end{equation*}
This observation motivates the concept of generalized pathwise rough entropy/kinetic solutions introduced next. 
%in Section \ref{sec:generalized_basic}. As already outlined above, the existence of generalized rough kinetic  solutions to \eqref{eq:scl} will be a consequence of the new the energy bounds that will be derived below.
%The last step of the proof of existence of pathwise entropy solutions we prove that every generalized pathwise entropy solution gives rise to a pathwise %entropy solution in Section \ref{sec:reconstruction}. In conclusion, the concept of generalized pathwise entropy solutions will be relevant for us only as an %intermediate step in the construction of pathwise entropy solutions, dealing with the lack of strong convergence of the approximants.

\subsection*{Generalized pathwise rough entropy/kinetic 
solutions and some basic properties}\label{sec:generalized_basic}

We begin with the following definition. %and note, in advance, that a pathwise rough entropy/kinetic solution clearly sa

\begin{defn}
\label{def:path_e-soln-1}Assume that  $u_{0}\in(L^{1}\cap L^{2})(\R^{N})$. Then $f\in L^{\infty}([0,T];L^{1}(\R^{N}\times\R))\cap L^{\infty}(\R^{N}\times\R\times[0,T])$ is a generalized pathwise rough entropy/kinetic solution to \eqref{eq:scl} if there exists a nonnegative measure $\nu$ and a nonnegative, bounded measure $m$ on $\R^{N}\times\R\times[0,T]$ such that 
\begin{equation}\label{eq:gen_kinetic_measure}
%\begin{cases}
f(x,\xi,0)=\chi(u_{0}(x),\xi),  \  |f|(x,\xi,t)=\sgn(\xi)f(x,\xi,t)\le1  \ \text{ and } \  \frac{\partial f}{\partial\xi} =\d(\xi)-\nu(x,\xi,t),
%\end{cases}
\end{equation}
and, for all $t_{0}\ge0,$ all test functions $\vr_{t_{0}}$ given by \eqref{eq:transport_stable} with $\vr^{0}\in C^{\infty}_c (\R^{N+1})$ and all $\vp\in C_{c}^{\infty} ([0,T]),$ %, %where $f(x,\xi,0)=\chi(u_{0}(x),\xi)$ and
$$\int_{0}^{T}\partial_{t}\vp(r)(\vr_{t_{0}}\ast f)(y,\eta,r)dr+\vp(0)(\vr_{t_{0}}\ast f)(y,\eta,0)=\int_{0}^{T}\int\vp(r)\partial_{\xi}\vr_{t_{0}}(x,y,\xi,\eta,r)m(x,\xi,t)dxd\xi dr.$$
\end{defn}
%\begin{align*}
% & \int_{0}^{T}\partial_{t}\vp(r)(\vr_{t_{0}}\ast f)(y,\eta,r)dr+\vp(0)(\vr_{t_{0}}\ast f)(y,\eta,0)\\
% & =\int_{0}^{T}\int\vp(r)\partial_{\xi}\vr_{t_{0}}(x,y,\xi,\eta,r)m(x,\xi,t)dxd\xi dr
%\end{align*}
%for all $\vp\in C_{c}^{\infty}([0,T))$ , where $f(x,\xi,0)=\chi(u_{0}(x),\xi)$ and
%\begin{align}
%|f| & (x,\xi,t)=\sgn(\xi)f(x,\xi,t)\le1\label{eq:gen_kinetic_measure}\\
%\frac{\partial f}{\partial\xi} & =\d(\xi)-\nu(x,\xi,t)\nonumber 
%\end{align}
%for some nonnegative measure $\nu$.
%\end{defn}
\no It is, of course, immediate that a pathwise rough entropy/kinetic solution is also a generalized one.  Moreover, a  claim similar to Proposition~\ref{rmk:rough_kinetic} is true here too. Indeed  $f$ is a generalized pathwise rough entropy/kinetic  solution to \eqref{eq:scl} if and only if  there is a null $\mcN\subseteq[0,T]$ such that $0\not\in\mcN$ and, for all $t_{0}\ge0$, all $(s,t)\in \D\left([0,T]\setminus\mcN\right)$
 and all test functions $\vr_{t_{0}}$ given by \eqref{eq:transport_stable} with $\vr^{0}\in C^{\infty}_c(\R^{N+1})$, %we have
\begin{equation}
\vr_{t_{0}}\ast f(y,\eta, t) -\vr_{t_{0}}\ast f(y,\eta,s)=  -\int_{s}^{t}\int\partial_{\xi}\vr_{t_{0}}(x,y,\xi,\eta,r)m(x,\xi,r)dxd\xi dr.
\end{equation}\label{eq:gen_kinetic_integrated}
%\end{align}

\no We show first %prove %derive some basic estimates and properties of the  generalized generalized pathwise entropy %rough kinetic 
%solutions. In particular, we prove 
that the solutions given by Definition~\ref{def:path_e-soln-1} %rough kinetic solutions are 
are $L^\infty(\R^N\times \R) \ \text{weak} \ \star$  continuous at $t=0$. 
\begin{prop}
Let $f$ be a generalized pathwise rough entropy/kinetic solution %rough kinetic 
to \eqref{eq:scl}. Then there exists  a null $\mcN\subseteq[0,T]$ such that, if $t_{n}\in[0,T]\setminus\mcN$ and $t_{n}\to0$ as $n \to \infty,$
\[
f(x,\xi,t_{n})\rightharpoonup \chi(\xi,u_{0}(x)) \ \text{in}  \ L^{\infty}(\R^{N}\times\R) \ \text{weak} \star.
\]
%for $n\to\infty.$
\end{prop}
\begin{proof}
The argument is similar to  the one in \cite[Proposition 4.1.7]{P02}. Since $f\in L^{\infty}(\R^{N}\times\R\times[0,T])$ and $m$ is a finite measure on $\R^{N}\times\R\times[0,T],$ there exists a null $\mcN\subseteq[0,T]$ with the property that, every sequence  $t_{n}\in[0,T]\setminus\mcN$ with $t_{n}\to0$ has a subsequence (again denoted by $t_{n}$) such that, as $n\to \infty,$
\begin{align*}
f(x,\xi,t_{n}) & \rightharpoonup F(x,\xi) \ \text{in }L^{\infty}(\R^{N}\times\R) \ \text{ weak} \ \star, \\
\int_{0}^{t_{n}}m(x,\xi,r)dr & \rightharpoonup M(x,\xi) \text{ in } \  \mcM(\R^{N}\times\R)   \ \text{ weak} \ \star,
\end{align*}
and
\begin{align*}
\sgn(\xi)F(x,\xi) & =|F|(x,\xi)\le1,\\
M(x,\xi)\ge0 & \ \text{ and } \  \int M(x,\xi)dxd\xi\le\int_{0}^{T}\int m(x,\xi,r)dxd\xi dr.
\end{align*}
Since $f$ is a generalized pathwise rough entropy/kinetic solution %rough kinetic solution 
we have
\begin{align*}
(\vr_{0}\ast f)(y,\eta,t_{n}) - (\vr_{0}\ast f)(y,\eta,0) = & -\int_{0}^{t_{n}}\int\partial_{\xi}\vr_{0}(x,y,\xi,\eta,r)m(x,\xi,r)dxd\xi dr.
\end{align*}
Letting $t_{n}\to0$ yields
\begin{align*}
(\vr^{0}\ast F)(y,\eta)-(\vr^{0}\ast f^0)(y,\eta)= & -\int\partial_{\xi}\vr^{0}(x,y,\xi,\eta)M(x,\xi)dxd\xi,
\end{align*}
and, hence,
\begin{align*}
F(x,\xi)=\chi(\xi,u_{0}(x)) & +\partial_{\xi}M(x,\xi)
\end{align*}
in the sense of distributions. Then \cite[Lemma 2.2.3]{P02} yields  $M=0$ and $F(x,\xi,t)=\chi(\xi,u_{0}(x))$.

\end{proof}
\subsection*{Two important lemmata} We present here two technical lemmata about the behavior of certain expressions (integrals) involving $f$, $m$ and the special test functions we are using here. These results were already used in the proof of the contraction proof (uniquess) in the previous section and are also important for the existence. %The reason we look at the these quantities is that they appear and play prominent role in the existence and uniqueness proofs in the paper. 
At a first passage the reader may want to skip them and go straight to the next subsection.

\begin{lem}
\label{lem:abs_val_part}Let $f$  be a generalized pathwise rough entropy/kinetic 
solution. There exists a null $\mcN\subseteq[0,T]$ with $0\not\in\mcN$ such that, for every $t_{0}\ge0$, all $\vr_{t_{0}}$ as in \eqref{eq:transport_stable} and all $(s,t)\in \D\left([0,T]\setminus\mcN\right)$, 
\begin{align*}
& 2\int_{s}^{t}\int\vr_{t_{0}}(x',y,0,\eta,r)\vr_{t_{0}}(x,y,\xi,\eta,r)m(x,\xi,r)dxd\xi dx'dyd\eta dr\\
 & =-\int(\sgn\ast\vr^{0})(y,\eta)(f\ast\vr_{t_{0}})(y,\eta,t)dyd\eta+ \int(\sgn\ast\vr^{0})(y,\eta)(f\ast\vr_{t_{0}})(y,\eta,s)+\int_{s}^{t}Err(t_{0},r)dr,
\end{align*}
where
\begin{align*}
Err(t_{0},t)=-\int m(x,\xi,t)\sgn(\xi') ( & \partial_{\xi'}\vr_{t_{0}}(x',y,\xi',\eta,t)\vr_{t_{0}}(x,y,\xi,\eta,t)\\
 & +\vr_{t_{0}}(x',y,\xi',\eta,t)\partial_{\xi}\vr_{t_{0}}(x,y,\xi,\eta,t))dyd\eta dx'd\xi'dxd\xi.
\end{align*}
\end{lem}
\begin{proof}
For all $s\le t$ we have
\begin{align*}
 & 2\int_{s}^{t}\int\vr_{t_{0}}(x',y,0,\eta,r)\vr_{t_{0}}(x,y,\xi,\eta,r)m(x,\xi,r)dxd\xi dx'dyd\eta dr\\
 & =\int_{s}^{t}\int\partial_{\xi'}\sgn(\xi')\vr_{t_{0}}(x',y,\xi',\eta,r)\vr_{t_{0}}(x,y,\xi,\eta,r)m(x,\xi,r)dxd\xi dx'd\xi'dyd\eta dr\\
 & =\int_{s}^{t}\int\left(\int\sgn(\xi')\vr_{t_{0}}(x',y,\xi',\eta,r)dx'd\xi'\right)\left(\int\partial_{\xi}\vr_{t_{0}}(x,y,\xi,\eta,r)m(x,\xi,r)dxd\xi\right)dyd\eta dr\\
 & +\int_{s}^{t}Err(t_{0},r)dr.
\end{align*}
It follows from \eqref{eq:sign_char} that 
\begin{align*}
 & \int\sgn(\xi')\vr_{t_{0}}(x',y,\xi',\eta,t)dx'd\xi'\\
 & =\int\sgn(\Xi_{(t,x',\xi')}(t-t_{0}))\vr^{0}\left(\begin{array}{cc}
X_{(t,x',\xi')}(t-t_{0})-y\\
\Xi_{(t,x',\xi')}(t-t_{0})-\eta
\end{array}\right)dx'd\xi'\\
 & =\int\sgn(\xi')\vr^{0}\left(\begin{array}{cc}
x'-y\\
\xi'-\eta
\end{array}\right)dx'd\xi'\\
 & =(\sgn\ast\vr^{0})(y,\eta).
\end{align*}
Hence,  there exists a null  $\mcN\subseteq[0,T]$ with $0\not\in\mcN$ such that, for every $t_{0}\ge0$ and all $(s,t)\in\D\left([t_{0},T]\setminus\mcN\right),$
%we have 
\begin{align*}
 & 2\int_{s}^{t}\int\vr_{t_{0}}(x',y,0,\eta,r)\vr_{t_{0}}(x,y,\xi,\eta,r)m(x,\xi,r)dxd\xi dx'd\xi'dyd\eta dr\\
 & =\int_{s}^{t}\int(\sgn\ast\vr^{0})(y,\eta)\left(\int\partial_{\xi}\vr_{t_{0}}(x,y,\xi,\eta,r)m(x,\xi,r)dxd\xi\right)dyd\eta 
 +\int_{s}^{t}Err(t_{0},r)dr\\
 & =-\int(\sgn\ast\vr^{0})(y,\eta)(f\ast\vr_{t_{0}})(y,\eta,\cdot)|_{s}^{t}dyd\eta+\int_{s}^{t}Err(t_{0},r)dr.
\end{align*}
\end{proof}
\no As in the proof of Theorem \ref{thm:unique} let $\vr_{\ve}^{s},\vr_{\ve}^{v}$ be standard smooth approximations of Dirac masses and, for each $t_{0}\ge0$, let $\vr_{t_{0},\ve}(x,y,\xi,\eta,t)$ be as in \eqref{eq:transport_stable} with initial condition $\vr_{\ve}^{0}(x,y,\xi,\eta)=\vr_{\ve}^{s}(x-y)\vr_{\ve}^{v}(\xi-\eta)$, that is  $\vr_{t_{0},\ve}$ is as in \eqref{eq:eps_test_fct}.
%for all $t\in[0,T].$ 
\begin{lem}
\label{lem:convergence}Let $f$ be a generalized pathwise rough entropy/kinetic 
solution to \eqref{eq:scl}. There exists a null $\mcN\subseteq[0,T]$ with $0\not\in\mcN$ such that, for every $t_{0}\ge0$ and all $t\in[t_{0},T]\setminus\mcN$,  as $\ve\to0$%we have
\begin{align*}
f\ast\vr_{t_{0},\ve}(y,\eta,t) & \to f(Y_{(t_{0},y,\eta)}(t),\z_{(t_{0},y,\eta)}(t),t),
\end{align*}
where $\vr_{t_{0},\ve}$ is as in \eqref{eq:eps_test_fct}.
%for $\ve\to0$.
\end{lem}
\begin{proof}
The proof is immediate, since, for a.e. $t\in[t_{0},T]$, as  $\ve\to0$,
\begin{align*}
f\ast\vr_{t_{0},\ve}(y,\eta,t) & =\int f(x,\xi,t)\vr_{t_{0},\ve}(x,y,\xi,\eta,t)dxd\xi\\
 & =\int f(x,\xi,t)\vr_{\ve}^{s}(X_{(t,x,\xi)}(t-t_{0})-y)\vr_{\ve}^{v}(\Xi_{(t,x,\xi)}(t-t_{0})-\eta)dxd\xi\\
 & =\int f(Y_{(t_{0},x,\xi)}(t),\z_{(t_{0},x,\xi)}(t),t)\vr_{\ve}^{s}(x-y)\vr_{\ve}^{s}(\xi-\eta)dxd\xi\\
 & \to f(Y_{(t_{0},y,\eta)}(t),\z_{(t_{0},y,\eta)}(t),t), 
\end{align*}
%for a.e. $t\in[t_{0},T]$. 
and, in light of $f(x,\xi,0)=\chi(\xi,u_{0}(x))$,  this is true, in particular, for $t=0$.
\end{proof}

\subsection*{Stable a priori estimates}\label{sec:stable_apriori} We use here the main idea of the definition of the pathwise rough entropy/kinetic solution, namely the use of test functions that propagate along the characteristics of the kinetic equation, to improve  \eqref{takis2}.

\no We begin with the following preliminary result.

\begin{lem}
\label{lem:l1-bound}Let $f$ be a generalized pathwise rough entropy/kinetic  solution %rough kinetic solution 
to \eqref{eq:scl}.  There exists a null  $\mcN\subseteq[0,T]$ such that, for all $t\in[0,T]\setminus\mcN$, 
\begin{align*}
\int|f|(x,\xi,t)dxd\xi\le & \|u_{0}\|_{1}.
\end{align*}
%for all $t\in[0,T]\setminus\mcN$.
\end{lem}
\begin{proof}
Fix  $\vp\in C_{c}^{\infty}(\R)$  monotone with $ $$\vp(0)=0$ and  use $\vr_{t_{0}}=\vp(\Xi_{(t,x,\xi)}(t-t_{0}))$ as a test function  with $(y,\eta)=(0,0)$ to get a null $\mcN\subseteq[0,T]$ such that $0\not\in\mcN$ and, for all $t_0 \in [0,T]$, $(s,t) \in \D\left([t_0,T]\setminus\mcN\right),$
%\setminus\mcN$ for some zero set $\mcN\subseteq[0,T]$ with $0\not\in\mcN$.
\begin{align*}
\int f(x,\xi,t)\vp(\Xi_{(t,x,\xi)}(t-t_{0}))dxd\xi -& \int f(x,\xi,s)\vp(\Xi_{(s,x,\xi)}(s-t_{0}))dxd\xi   \\
& =-\int_{s}^{t}\int\partial_{\xi}\vp(\Xi_{(r,x,\xi)}(r-t_{0}))m(x,\xi,r)dxd\xi dr.
\end{align*}
%for all $t_{0}\le s\le t$, $s,t\in[0,T]\setminus\mcN$ for some zero set $\mcN\subseteq[0,T]$ with $0\not\in\mcN$. 
Choosing $s=t_{0}\not\in\mcN$ we find
\begin{align}
\int_{s}^{t}\int \partial_{\xi}\vp(\Xi_{(r,x,\xi)}(r-s))\partial_{\xi}\Xi_{(r,x,\xi)}(r-s)m(x,\xi,r)d\xi dxdr\nonumber \\
+\int f(x,\xi,t)\vp(\Xi_{(t,x,\xi)}(t-s))dxd\xi\le & \int f(x,\xi,s)\vp(\xi)dxd\xi.\label{eq:l1_1}
\end{align}
Since 
\[
\sgn(\Xi_{(t,x,\xi)}(t-s))=\sgn(f(x,\xi,t))=\sgn(\xi),
\]
it follows that, for all  $t\in[t_{0},T]\setminus\mcN$ and a.e. in $x$ and $\xi$,
\[
f(x,\xi,t)\vp(\Xi_{(t,x,\xi)}(t-t_{0}))\ge0.%\quad\text{a.e.}
\]
%for all $t\in[t_{0},T]\setminus\mcN$. 
In view of Lemma \ref{lem:hoelder_rp}, we know that there exists $C=C(R)>0$ such that,  for all $\a$-Hölder rough paths $z$ with $\|z\|_{\a-\Hoel;[0,T]}\le R$ and all $r\in[0,t]$,
\begin{equation}\label{eq:der_stability}
|\partial_{\xi}\Xi_{(r,x,\xi)}(t)-1|=|\partial_{\xi}\Xi_{(r,x,\xi)}(t)-\partial_{\xi}\Xi_{(r,x,\xi)}(0)|\le\|\partial_{\xi}\Xi_{(r,x,\xi)}(t)\|_{\a-\Hoel;[0,r]}|t|^{\a} \le C|t|^{\a}.
\end{equation}
%\begin{align}
%|\partial_{\xi}\Xi_{(r,x,\xi)}(t)-1| & =|\partial_{\xi}\Xi_{(r,x,\xi)}(t)-\partial_{\xi}\Xi_{(r,x,\xi)}(0)|\nonumber \\
% & \le\|\partial_{\xi}\Xi_{(r,x,\xi)}(t)\|_{\a-\Hoel;[0,r]}|t|^{\a}\label{eq:der_stability}\\
% & \le C|t|^{\a},\nonumber 
%\end{align}
%for all $t\in[0,r]$, where $C=C(R)$ may be chosen uniformly for all $\a$-Hölder rough paths $z$ with $\|z\|_{\a-\Hoel;[0,T]}\le R$. 
Hence, for $h>0$ small enough and all $|t-s|\le h$, we have
\begin{align*}
\inf_{r\in[s,t]}\partial_{\xi}\Xi_{(r,x,\xi)}(r-s)\ge & 0,
\end{align*}
%for all $|t-s|\le h$. 
and,  hence,
\[
\int_{s}^{t}\int \partial_{\xi}\vp \Xi_{(r,x,\xi)}(r-s)(\partial_{\xi}\Xi_{(r,x,\xi)}(r-s)m(x,\xi,r)d\xi dxdr \geq 0,
\]
and, from \eqref{eq:l1_1}, for all $s,t\in\D\left([0,T]\setminus\mcN\right)$ with $|t-s|\le h$,
\begin{align*}
\int f(x,\xi,t)\vp(\Xi_{(t,x,\xi)}(t-s))dxd\xi\le & \int f(x,\xi,s)\vp(\xi)dxd\xi.
\end{align*}
%for all $s\le t$, $s,t\in[0,T]\setminus\mcN$ with $|t-s|\le h$.

\no We now extend the previous estimate to $\vp:\R\to\R$ which are  measurable and monotone with $\vp(0)=0$. Given such $\vp$ let  $\vp^{n}\in C_{c}^{\infty}(\R)$ be monotone functions  such that  $\vp^{n}(0)=0$ and $\vp^{n}\to\vp$ pointwise.  It follows from Fatou's Lemma that
\begin{align}
\int f(x,\xi,t)\vp(\Xi_{(t,x,\xi)}(t-s))dxd\xi\le & \liminf_{n\to\infty}\int f(x,\xi,s)\vp^{n}(\xi)dxd\xi.\label{eq:l1_2}
\end{align}

\no Next observe that we may choose $M=M(R)>0$, uniformly for all $\a$-Hölder rough paths $z$ with $\|z\|_{\a-\Hoel;[0,T]}\le R$ and, in view of \eqref{eq:der_stability}, for each $t\in[0,T]\setminus\mcN$, a partition $\D=\{0=t_{0}\le\dots\le t_{M}=t\}\subseteq[0,T]\setminus\mcN$ such that, for 
all $ i=0,\dots,M-1,$
\[
\inf_{r\in[t_{i},t_{i+1}]}\partial_{\xi}\Xi_{(r,x,\xi)}(r-s)\ge0.%\quad\forall i=0,\dots,M-1.
\]
%and $M=M(R)$ may be chosen uniformly for all $\a$-Hölder rough paths $z$ with $\|z\|_{\a-\Hoel;[0,T]}\le R$.

\no The claim now follows from an elementary iterative argument: Choosing $\vp=\sgn$, $\vp^{n}$ such that $|\vp^{n}|\le|\vp|$ and $(s,t)=(0,t_{1})$ in \eqref{eq:l1_2} we find 
$$\int f(x,\xi,t_{1})\sgn(\Xi_{(t_{1},x,\xi)}(t_{1}))dxd\xi\le  \int f(x,\xi,0)\sgn(\xi)dxd\xi
= \int|f|(x,\xi,0)dxd\xi 
=  \|u_{0}\|_{1}^{1},$$
and, since $\sgn(\Xi_{(t_{1},x,\xi)}(t_{1}))=\sgn(\xi)$, %this yields
%\begin{align*}
$$\int|f|(x,\xi,t_{1})dxd\xi\le  \int|f|(x,\xi,0)dxd\xi
=  \|u_{0}\|_{1}^{1}.$$
%\end{align*}

\no Iterating this argument over $i$ yields, for all $t\in[0,T]\setminus\mcN$,
\begin{align*}
\int|f|(x,\xi,t)dxd\xi \le & \|u_{0}\|_{1}^{1}.
\end{align*}
%for all $t\in[0,T]\setminus\mcN$.
\end{proof}

\no We present now the estimate on the total mass which is new even in the class of homogeneous conservation laws without rough time dependence.

\begin{lem} \label{lem:l2-bound}Let $f$ be a generalized pathwise rough entropy/kinetic 
solution to \eqref{eq:scl}. There exists a null  $\mcN\subseteq[0,T]$ and, for each $R>0$,  $M=M(R)>0$ such that, for all driving signals ${z}$ with $\|{z}\|_{\a-\Hoel;[0,T]}\le R$ and   all $t\in[0,T]\setminus\mcN$,
\begin{align*}
\frac{1}{2}\int_{0}^{t}\int m(x,\xi,r)d\xi dxdr+\int f(x,\xi,t)\xi dxd\xi\le & \frac{1}{2}\|u_{0}\|_{2}^{2}+M\|u_{0}\|_{1}.
\end{align*}
%for all $t\in[0,T]\setminus\mcN$ for some constant $M>0$. The constant $M=M(R)$ can be chosen uniformly for driving signals $z$ with $\|z\|_{\a-\Hoel;[0,T]}\le R$.
\end{lem}
\begin{proof}
In view of  \eqref{eq:l1_1},  for $\vp\in C_{c}^{\infty}(\R)$  monotone with $\vp(0)=0$, a null set  $\mcN\subseteq[0,T]$ with $0\not\in\mcN$ and for all  $(s,t)\in\D\left([0,T]\setminus\mcN\right)$, we have 
\begin{align*}
\int_{s}^{t}\int\partial_{\xi}\vp(\Xi_{(r,x,\xi)}(r-s))\partial_{\xi}\Xi_{(r,x,\xi)}(r-s)m(x,\xi,r)d\xi dxdr\\
+\int f(x,\xi,t)\vp(\Xi_{(t,x,\xi)}(t-s))dxd\xi\le & \int f(x,\xi,s)\vp(\xi)dxd\xi.
\end{align*}
Moreover, in light of  \eqref{eq:der_stability}, for $h>0$ small enough and all for all $|t-s|\le h$, we have
\begin{align*}
\inf_{r\in[s,t]}\partial_{\xi}\Xi_{(r,x,\xi)}(r-s)\ge & \frac{1}{2}.
\end{align*}
%for all $|t-s|\le h$. 
Hence, for all $s,t\in\D\left([0,T]\setminus\mcN\right)$ with $|t-s|\le h$,
\begin{align*}
\frac{1}{2}\int_{s}^{t}\int\partial_{\xi}\vp(\Xi_{(r,x,\xi)}(r-s))m(x,\xi,r)d\xi dxdr\\
+\int f(x,\xi,t)\vp(\Xi_{(t,x,\xi)}(t-s))dxd\xi\le & \int f(x,\xi,s)\vp(\xi)dxd\xi,
\end{align*}
%for all $s\le t$, $s,t\in[0,T]\setminus\mcN$ with $|t-s|\le h$.

\no Let $\vp\in C^{1}(\R)$ be monotone with $\vp(0)=0$ and choose  $\vp^{n}\in C_{c}^{\infty}(\R)$ monotone with $\vp^{n}(0)=0$ and such that $\vp^{n}\to\vp$ and $\partial_{\xi} \vp^{n} \to \partial_{\xi} \vp$ pointwise. Again Fatou's Lemma yields,  for all $s,t\in\D\left([0,T]\setminus\mcN\right)$ with $|t-s|\le h$,
\begin{align*}
\frac{1}{2}\int_{s}^{t}\int(\partial_{\xi}\vp)(\Xi_{(r,x,\xi)}(r-s))m(x,\xi,r)d\xi dxdr\\
+\int f(x,\xi,t)\vp(\Xi_{(t,x,\xi)}(t-s))dxd\xi\le & \liminf_{n\to\infty}\int f(x,\xi,s)\vp^{n}(\xi)dxd\xi.
\end{align*}
%for all $s\le t$, $s,t\in[0,T]\setminus\mcN$, $|t-s|\le h$. 
Letting  $\vp(\xi)=\xi$ in the inequality above gives
\begin{align*}
\frac{1}{2}\int_{s}^{t}\int m(x,\xi,r)d\xi dxdr+\int f(x,\xi,t)\Xi_{(t,x,\xi)}(t-s)dxd\xi\le & \int f(x,\xi,s)\xi dxd\xi.
\end{align*}
Once again it follows from  Lemma \ref{lem:hoelder_rp} that, for each  $R>0$, there exists $C=C(R)>0$ such that, for all $\a$-Hölder rough paths ${z}$ with $\|{z}\|_{\a-\Hoel;[0,T]}\le R$ and all $r\in [0,t]$, 
%\begin{align}
\begin{equation}\label{eq:der_stability-1}
|\Xi_{(r,x,\xi)}(t)-\xi|  =|\Xi_{(r,x,\xi)}(t)-\Xi_{(r,x,\xi)}(0)|
  \le\|\Xi_{(r,x,\xi)}(t)\|_{\a-\Hoel;[0,r]}|t|^{\a}  \le C|t|^{\a}.
\end{equation}
%for all $t\in[0,r]$, where $C=C(R)$ may be chosen uniformly for all $\a$-Hölder rough paths $z$ with $\|z\|_{\a-\Hoel;[0,T]}\le R$. 

\no In light of  \eqref{eq:der_stability} and \eqref{eq:der_stability-1}, for each $R>0$ we can find $M=M(R)>0$ such that, for all $\a$-Hölder rough paths ${z}$ with $\|{z}\|_{\a-\Hoel;[0,T]}\le R$ and  for each $t\in[0,T]\setminus\mcN$, there exists   a partition $\D=\{0=t_{0}\le\dots\le t_{M}=t\}\subseteq[0,T]\setminus\mcN$ such that $ \|\D\|  \le1$ and, for all $i=0,\dots,M-1$,
\begin{align*}
\inf_{r\in[t_{i},t_{i+1}]}\partial_{\xi}\Xi_{(r,x,\xi)}(r-t_{i})  \ge\frac{1}{2}\quad \text{and} 
\sup_{r\in[t_{i},t_{i+1}]}|\Xi_{(r,x,\xi)}(r-t_{i})-\xi|  \le1. %\quad\forall i=0,\dots,M-1,\\
%\|\D\| & \le1
\end{align*}
% and $M=M(R)$ may be chosen uniformly for all $\a$-Hölder rough paths $z$ with $\|z\|_{\a-\Hoel;[0,T]}\le R$.

\no Observe that 
%\begin{align*}
$$\frac{1}{2}\int_{0}^{t_{1}}\int m(x,\xi,r)d\xi dxdr+\int f(x,\xi,t_{1})\Xi_{(t_{1},x,\xi)}(t_{1})dxd\xi\le  \int f(x,\xi,0)\xi dxd\xi
=  \frac{1}{2}\|u_{0}\|_{2}^{2}.$$
%\end{align*}
Moreover, using Lemma \ref{lem:l1-bound}, we obtain 
\begin{align*}
\int f(x,\xi,t_{1})\xi dxd\xi= & \int f(x,\xi,t_{1})\Xi_{(t_{1},x,\xi)}(t_{1})dxd\xi
  +\int f(x,\xi,t_{1})(\xi-\Xi_{(t_{1},x,\xi)}(t_{1}))dxd\xi\\
\le & \int f(x,\xi,t_{1})\Xi_{(t_{1},x,\xi)}(t_{1})dxd\xi+\int|f|(x,\xi,t_{1})dxd\xi\\
\le & \int f(x,\xi,t_{1})\Xi_{(t_{1},x,\xi)}(t_{1})dxd\xi+\|u_{0}\|_{1}.
\end{align*}
Thus 
\begin{align*}
\frac{1}{2}\int_{0}^{t_{1}}\int m(x,\xi,r)d\xi dxdr+\int f(x,\xi,t_{1})\xi dxd\xi\le & \int f(x,\xi,0)\xi dxd\xi+\|u_{0}\|_{1}.
\end{align*}
Iterating the above  argument over $i$ yields
\begin{align*}
\frac{1}{2}\int_{0}^{t}\int m(x,\xi,r)d\xi dxdr+\int f(x,\xi,t)\xi dxd\xi\le & \int f(x,\xi,0)\xi dxd\xi+M\|u_{0}\|_{1}\\
= & \frac{1}{2}\|u_{0}\|_{2}^{2}+M\|u_{0}\|_{1}.
\end{align*}

\end{proof}

\subsection*{The construction of generalized pathwise rough entropy/kinetic solutions}\label{sec:generalized_construction} The following theorem asserts the existence 
of the generalized pathwise rough entropy/kinetic solutions.

\begin{thm}\label{thm:generalized_existence}
For each   $u_{0}\in(L^{1}\cap L^{2})(\R^{N})$ there exists a generalized pathwise rough entropy/kinetic solution $f\in L^{\infty}([0,T];L^{1}(\R^{N}\times\R))\cap L^{\infty}(\R^{N}\times\R\times[0,T])$ to \eqref{eq:scl}.\end{thm} 
\begin{proof}
Given a geometric rough path ${ z}$ choose smooth  ${z}^{n}$ such that   ${z}^{n}\to {z}$ in the rough paths metric, that is   in $C^{0,\beta-\text{H\"ol}}_0([0,T],G^{\lfloor 1/\beta \rfloor}(\R^N))$, % with $z^{n}$ smooth 
and consider the unique kinetic solutions $(u^{n},m^{n})$ to \eqref{eq:scl}, that is  $\chi^{n}=\chi(u^{n})$ solving \eqref{eq:kinetic_form-1-4-1}. 

\no We use  Lemma \ref{lem:kinetic-pathwise} with  $t_{0}\ge0$ and test functions $\vr_{t_{0}}^{n}$ given by \eqref{eq:transport} (with $z$ replaced by $z^{n}$) and  $\vr^{0}\in C_{c}^{\infty}(\R^{N+1})$ to find, for all $\vp\in C_{c}^{\infty}([0,T))$,
\begin{align}\label{eq:approx_rough_kinetic_1}
\int_{0}^{T}\partial_{t}\vp(r)(\vr_{t_{0}}^{n}\ast\chi^{n})(y,\eta,r)dr+ &\vp(0)(\vr_{t_{0}}^{n}\ast\chi^{n})(y,\eta,0)\\ 
& =\int_{0}^{T}\int\vp(r)\partial_{\xi}\vr_{t_{0}}^{n}(x,y,\xi,\eta,r)m^{n}(x,\xi,r)dxd\xi dr, \nonumber 
\end{align}
%\end{equation}
and, hence,
\begin{align}\label{eq:approx_rough_kinetic}
 & \int_{0}^{T}\int\partial_{t}\vp(r)\vr_{t_{0}}^{n}(x,y,\xi,\eta,r)\chi^{n}(x,\xi,r)dxd\xi dr+\int\vp(0)\vr_{t_{0}}^{n}(x,y,\xi,\eta,0)\chi(u_{0}(x),\xi)dxd\xi\\
 & =\int_{0}^{T}\int\vp(r)\partial_{\xi}\vr_{t_{0}}^{n}(x,y,\xi,\eta,r)m^{n}(x,\xi,r)dxd\xi dr. \nonumber 
\end{align}
Since,  for all $t\in[0,T],$ 
\[
\vr_{t_{0}}^{n}(x,y,\xi,\eta,t)=\vr^{0}\left(\begin{array}{cc}
X_{(t,x,\xi)}^{n}(t-t_0)-y\\
\Xi_{(t,x,\xi)}^{n}(t-t_0)-\eta
\end{array}\right),%\quad\forall t\in[0,T]
\]
and, as $n\to\infty,$ (see Appendix \ref{app:RP})
\[
\sup_{x,\xi}\left\Vert \left(\begin{array}{c}
X_{(\cdot,x,\xi)}^{n}(\cdot-t_0)\\
\Xi_{(\cdot,x,\xi)}^{n}(\cdot-t_0)
\end{array}\right)-\left(\begin{array}{c}
X_{(\cdot,x,\xi)}(\cdot-t_0)\\
\Xi_{(\cdot,x,\xi)}(\cdot-t_0)
\end{array}\right)\right\Vert _{C^{0}([0,T])}\to0,%\quad\text{for }n\to\infty.
\]
we have, as $n\to \infty$ and  uniformly in $(x,\xi,t)$,
\[
\vr_{t_{0}}^{n}(x,y,\xi,\eta,t)\to\vr_{t_{0}}(x,y,\xi,\eta,t)
\]
%uniformly in $(x,\xi,t)$. 

%$$\sup_{x,\xi} \| (X_{(\cdot,x,\xi)}^{n}(\cdot-t_0)-X_{(\cdot,x,\xi)}(\cdot-t_0)), \Xi_{(\cdot,x,\xi)}^{n}(\cdot-t_0)-\Xi_{(\cdot,x,\xi)}(\cdot-t_0))\| \to 0$$

\no Moreover, standard rough path  estimates (see Appendix \ref{app:RP}) yield $C>0$ such that  
\begin{equation}
\sup_{x,\xi}\left\Vert \left(\begin{array}{c}
X_{(\cdot,x,\xi)}^{n}(\cdot-t_0)-x\\
\Xi_{(\cdot,x,\xi)}^{n}(\cdot-t_0)-\xi
\end{array}\right)\right\Vert _{C^{0}([0,T])}\le C. \label{eq:unif_bound}
\end{equation}
%where $C$ is uniform in $n$. 

\no An immediate consequence of \eqref{eq:unif_bound} and the fact that  $\vr^{0}$ has compact support is that, for each $(y,\eta)$, the $\vr_{t_{0}}^{n}$'s also  have uniform in $n$  compact support, that is  there exists a compact  $K\subseteq\R^{N}\times\R$ such that, for all $(x,\xi)\not\in K$, 
\[
\vr_{t_{0}}^{n}(x,y,\xi,\eta,t)=0.%\quad\forall(x,\xi)\not\in K.
\]
In addition Lemma~\ref{lem:l1-bound} yields
\[
\sup_{t\in[0,T]\setminus\mcN}\|\chi^{n}(\cdot,\cdot,t)\|_{L^{1}(\R^{N}\times\R)}\le\|u_{0}\|_{1}.
\]
Using $|\chi^{n}|\le1$ this implies that, along a subsequence, which for simplicity is denoted by $\chi^{n}$,
$$\chi^{n}\rightharpoonup f\quad\text{in }L^{\infty}(\R^{N}\times\R\times[0,T]) \text{ weak } \star \ \text{ and} \  \|f\|_{L^{\infty}([0,T];L^{1}(\R^{N}\times\R))}\le\|u_{0}\|_{1}.$$
%and
%\[
%\|f\|_{L^{\infty}([0,T];L^{1}(\R^{N}\times\R))}\le\|u_{0}\|_{1}.
%\]
It then follows from \cite[Lemma 2.3.1]{P02} that there exists  a non-negative measure $\nu$ such that 
\begin{align*}
|f|(x,\xi,t)  =\sgn(\xi)f(x,\xi,t)\le1 \ \ \text{and } \ \ 
\partial_{\xi}f(x,\xi,t)  =\d(\xi)-\nu(t,x,\xi).
\end{align*}
%for a non-negative measure $\nu$. 

\no Next we use Lemma \ref{lem:l2-bound}, with  $\int\chi^{n}(x,\xi,t)\xi dxd\xi=\frac{1}{2}\|u^{n}(t)\|_{2}^{2}$, to get,  for all $t\in[0,T]\setminus\mcN$,
\begin{align*}
\frac{1}{2}\int_{0}^{t}\int m^{n}(x,\xi,r)d\xi dxdr+\frac{1}{2}\|u^{n}(t)\|_{2}^{2}\le & \frac{1}{2}\|u_{0}\|_{2}^{2}+M\|u_{0}\|_{1};
\end{align*}
%for all $t\in[0,T]\setminus\mcN$. 
note that $M$ may be chosen uniformly in $n$ since $\|{z}^{n}\|_{\a-\Hoel;[0,T]}\le R$ uniformly for some $R>0$. 

\no It follows that  there exists  a weak $\star $ convergent subsequence $m^{n}\rightharpoonup m$ and taking the limit in \eqref{eq:approx_rough_kinetic} yields
\begin{align*}
 & \int_{0}^{T}\int\partial_{t}\vp(r)\vr_{t_{0}}(x,y,\xi,\eta,r)f(x,\xi,r)dxd\xi dr+\int\vp(0)\vr_{t_{0}}(x,y,\xi,\eta,0)\chi(u_{0}(x),\xi)dxd\xi\\
 & =\int_{0}^{T}\int\vp(r)\partial_{\xi}\vr_{t_{0}}(x,y,\xi,\eta,r)m(x,\xi,r)dxd\xi dr.
\end{align*}

\no Hence, $f$ is a generalized pathwise rough entropy/kinetic solution %rough kinetic solution 
to \eqref{eq:scl}.
%, and, in view of Proposition \ref{prop:reconstruction},  there exists a rough kinetic solution $u\in L^{\infty}([0,T];L^{1}(\R^{N}))$ to \eqref{eq:scl}.
\end{proof}

\subsection*{The reconstruction of a pathwise rough entropy/kinetic solution from a generalized one}
%\texorpdfstring{$u$}{u} from generalized pathwise entropy solutions}\label{sec:reconstruction}
We show here how to go from generalized to  pathwise rough entropy/kinetic  solution.
\begin{prop}
\label{prop:reconstruction}Let $f$ be a generalized pathwise rough entropy/kinetic %rough kinetic 
solution to \eqref{eq:scl}. There exists  
\newline $u\in L^{\infty}([0,T];L^{1}(\R^{N}))$ such that, a.e. in $(x,\xi,t)$, $f(x,\xi,t)=\chi(\xi,u(x,t)).$
%\[
%f(x,\xi,t)=\chi(\xi,u(x,t)).
%\]
%almost everywhere. 

\end{prop}
\begin{proof}
The claim follows, if we show that, for a given generalized pathwise rough entropy/kinetic
solution, there exists a null set  $\mcN\subseteq[0,T]$ with $0\not\in\mcN$ such that, for all $t \in [0,T]\setminus \mcN$,
\begin{equation}\label{takis}
\int (f^{2}(y,\eta,t)-|f|(y,\eta,t))dyd\eta =0, %\quad\forall t\in[0,T]\setminus\mcN,
\end{equation}
which, in view of  \eqref{eq:gen_kinetic_measure},  implies that $f=\chi(u)$ for some measurable $u:[0,T]\times\R^{N}\to\R$. 

\no Since
%\begin{align*}
%\|u(t)\|_1%{L^{1}(\R^{N})} 
%& =\int|\chi|(\xi,u(t,x))d\xi dx\\
% & =\int|f|(x,\xi,t)d\xi dx\\
% & =\|f(t)\|_{L^{1}(\R^{N}\times\R)}
%\end{align*}
$$\|u(\cdot, t)\|_{L^{1}(\R^{N})} =\int|\chi|(\xi,u(x,t))d\xi dx =\int|f|(x,\xi,t)d\xi dx =\|f(\cdot, t)\|_{L^{1}(\R^{N}\times\R)},$$ 
it then follows that $u\in L^{\infty}([0,T];L^{1}(\R^{N})).$

\no In view of the obvious inequality $\int (f^{2}(y,\eta,t)-|f|(y,\eta,t))dyd\eta \leq 0$, to prove \eqref{takis} we have to show that 
\begin{equation}
\int (f^{2}(y,\eta,t)-|f|(y,\eta,t))dyd\eta \geq 0. \label{eq:diff}
\end{equation}
%from below, the aim being to show that the lower bound is $0$.

\no  Let $\vr_{t_{0},\ve}$ be as in \eqref{eq:eps_test_fct}. From Lemma \ref{lem:convergence} we know that there exists  a null $\mcN\subseteq[0,T]$ with $0\not\in\mcN$ such that, for all $t\in[0,T]\setminus\mcN$, as  $\ve\to0$,%we have 
\begin{align*}
f\ast\vr_{t_{0},\ve}(y,\eta,t) & \to f(Y_{(t_{0},y,\eta)}(t),\z_{(t_{0},y,\eta)}(t),t),
\end{align*}
%for $\ve\to0$. 
and, hence, for all $t\in[0,T]\setminus\mcN$, 
%\no We approximate \eqref{eq:diff} via (in what follows for simplicity we write $X_{(x,\xi)}$ for $X_{(t,x,\xi)}(t-t_{0})$)
\begin{align*}
  \int (f^{2}(y,\eta,t)-|f|(y,\eta,t))dyd\eta \
= & \int f^{2}(Y_{(t_{0},y,\eta)}(t),\z_{(t_{0},y,\eta)}(t),t)dyd\eta-\int\sgn(\eta)f(y,\eta,t)dyd\eta\\
= & \lim_{\ve\to0}\int(f\ast\vr_{t_{0},\ve})^{2}(y,\eta,t)dyd\eta-\int\sgn(\eta)f(y,\eta,t)dyd\eta.
\end{align*}
%for all $t\in[0,T]\setminus\mcN$. 
\no Thus to show  \eqref{eq:diff} it suffices to obtain an appropriate lower bound for 
\begin{equation}\label{takis1}
\int(f\ast\vr_{t_{0},\ve})^{2}(y,\eta,t)dyd\eta-\int\sgn(\eta)f(y,\eta,t)dyd\eta.
\end{equation}

\no To estimate the first term we note, that, for $t \in (t_0,T)$,
\begin{align*}
 \partial_{t}\int(f\ast\vr_{t_{0},\ve})^{2}dyd\eta
= & 2\int(f\ast\vr_{t_{0},\ve})\partial_{t}(f\ast\vr_{t_{0},\ve})dyd\eta\\
= & -2\int(f\ast\vr_{t_{0},\ve})(y,\eta,t)\partial_{\xi}\vr_{t_{0},\ve}(x,y,\xi,\eta,t)m(x,\xi,t)dxd\xi dyd\eta\\
= & -2\int f(x',\xi',t)\vr_{t_{0},\ve}(x',y,\xi',\eta,t)\partial_{\xi}\vr_{t_{0},\ve}(x,y,\xi,\eta,t)m(x,\xi,t)dx'd\xi'dxd\xi dyd\eta\\
= & 2\int f(x',\xi',t)\partial_{\xi'}\vr_{t_{0},\ve}(x',y,\xi',\eta,t)\vr_{t_{0},\ve}(x,y,\xi,\eta,t)m(x,\xi,t)dx'd\xi'dxd\xi dyd\eta\\
 & +Err^{(1)}(t_{0},t,\ve), %\quad\text{on }(t_{0},T),
\end{align*}
where
\begin{align*}
Err^{(1)}(t_{0},t,\ve)= & -2\int f(x',\xi',t)m(x,\xi,t)[\vr_{t_{0},\ve}(x',y,\xi',\eta,t)\partial_{\xi}\vr_{t_{0},\ve}(x,y,\xi,\eta,t)\\
 & +\partial_{\xi'}\vr_{t_{0},\ve}(x',y,\xi',\eta,t)\vr_{t_{0},\ve}(x,y,\xi,\eta,t)]dx'd\xi'dxd\xi dyd\eta.
\end{align*}
Now
\begin{align*}
 & 2\int f(x',\xi',t)\partial_{\xi'}\vr_{\ve}(x',y,\xi',\eta,t)\vr_{\ve}(x,y,\xi,\eta,t)m(x,\xi,t)dx'd\xi'dxd\xi dyd\eta\\
 & =-2\int(\d(\xi')-\nu(x',\xi',t))\vr_{\ve}(x',y,\xi',\eta,t)\vr_{\ve}(x,y,\xi,\eta,t)m(x,\xi,t)dx'd\xi'dxd\xi dyd\eta\\
 & =-2\int\vr_{\ve}(x',y,0,\eta,t)\vr_{\ve}(x,y,\xi,\eta,t)m(x,\xi,t)dx'dxd\xi dyd\eta\\
 & +2\int\nu(x',\xi',t)\vr_{\ve}(x',y,\xi',\eta,t)\vr_{\ve}(x,y,\xi,\eta,t)m(x,\xi,t)dx'd\xi'dxd\xi dyd\eta\\
 & \ge-2\int\vr_{\ve}(x',y,0,\eta,t)\vr_{\ve}(x,y,\xi,\eta,t)m(x,\xi,t)dx'dxd\xi dyd\eta.
\end{align*}
In view of  Lemma \ref{lem:abs_val_part}, there exists  a null  $\mcN\subseteq[0,T]$ with $0\not\in\mcN$ such that, for every $t_{0}\ge0$ and all $(s,t)\in\D\left([t_{0},T]\setminus\mcN\right),$
\begin{align*}
 & 2\int_{s}^{t}\int\vr_{t_{0},\ve}(x',y,0,\eta,r)\vr_{t_{0}}(x,y,\xi,\eta,r)m(x,\xi,r)dxd\xi dx'dyd\eta dr\\
 & =-\int\sgn^{\ve}(\eta)((f\ast\vr_{t_{0},\ve})(y,\eta,t) - (f\ast\vr_{t_{0},\ve})(y,\eta,s))dyd\eta+\int_{s}^{t}Err^{(2)}(t_{0},r,\ve)dr,
\end{align*}
with 
\begin{align*}
Err^{(2)}(t_{0},t,\ve)=\int m(x,\xi,t)\sgn(\xi')( & \partial_{\xi'}\vr_{t_{0},\ve}(x',y,\xi',\eta,t)\vr_{t_{0},\ve}(x,y,\xi,\eta,t)\\
 & +\vr_{t_{0},\ve}(x',y,\xi',\eta,t)\partial_{\xi}\vr_{t_{0},\ve}(x,y,\xi,\eta,t))dyd\eta dx'd\xi'dxd\xi.
\end{align*}
Hence, for all $t_{0}\ge0$ and all $s,t\in\D\left([t_{0},T]\setminus\mcN\right),$
\begin{align*}
&\int(f \ast\vr_{t_{0},\ve})^{2}(y,\eta,t)dyd\eta - \int(f\ast\vr_{t_{0},\ve})^{2}(y,\eta,s)dyd\eta 
\\\ge& \int\sgn^{\ve}(\eta)[(f\ast\vr_{t_{0},\ve})(y,\eta,t) - (f\ast\vr_{t_{0},\ve})(y,\eta,s)]dyd\eta
  +\int_{s}^{t}[Err^{(1)}(t_{0},r,\ve)-Err^{(2)}(t_{0},r,\ve)]dr.
%\int(f\ast\vr_{t_{0},\ve})^{2}dyd\eta|_{s}^{t}
\end{align*}
%for all $t_{0}\ge0$ and all $s\le t$, $s,t\in[t_{0},T]\setminus\mcN$. 
It follows from  Lemma \ref{lem:err_est}  that 
\begin{align*}
 & 
\int_{s}^{t}[Err^{(1)}(t_{0},r,\ve)-Err^{(2)}(t_{0},r,\ve)]dr\\
= & -2\int(f(x',\xi',t)+\sgn(\xi'))m(x,\xi,t)[\vr_{t_{0},\ve}(x',y,\xi',\eta,t)\partial_{\xi}\vr_{t_{0},\ve}(x,y,\xi,\eta,t)\\
 & +\partial_{\xi'}\vr_{t_{0},\ve}(x',y,\xi',\eta,t)\vr_{t_{0},\ve}(x,y,\xi,\eta,t)]dx'd\xi'dxd\xi dyd\eta\\
\ge & -4\int_{s}^{t}\int m(x,\xi,r)\Big(\int \Big|\int[\vr_{t_{0},\ve}(x',y,\xi',\eta,r)\partial_{\xi}\vr_{t_{0},\ve}(x,y,\xi,\eta,r)\\
 & +\partial_{\xi'}\vr_{t_{0},\ve}(x',y,\xi',\eta,r)\vr_{t_{0},\ve}(x,y,\xi,\eta,r)]dyd\eta \Big|dx'd\xi'\Big)dxd\xi dr\\
\ge & -C(t-t_{0})^{\a}\int_{s}^{t}\int m(x,\xi,r)dxd\xi dr,
\end{align*}
%\end{document}
and, thus, 
\begin{align*}
&\int[(f\ast\vr_{t_{0},\ve})^{2}(y,\eta,t) - (f\ast\vr_{t_{0},\ve})^{2}(y,\eta,s)]dyd\eta \\
&\ge  \int\sgn^{\ve}(\eta)[(f\ast\vr_{t_{0},\ve})(y,\eta,t) - f\ast\vr_{t_{0},\ve})(y,\eta,s)]dyd\eta-C(t-t_{0})^{\a}\int_{s}^{t}\int m(x,\xi,r)dxd\xi dr.
\end{align*}
Letting $\ve\to0$ and using \eqref{eq:gen_kinetic_measure} yields, for all $t_{0}\ge0$ and $(s,t)\in\D\left([t_{0},T]\setminus\mcN\right)$,
\begin{align*}
\int [f^{2}(y,\eta,t)-|f|(y,\eta,t)|]dyd\eta \ge & \int[f^{2}(y,\eta,s)-|f|(y,\eta,s)|]dyd\eta-C(t-t_{0})^{\a}\int_{s}^{t}\int m(x,\xi,r)dxd\xi dr.
\end{align*}
%for all $t_{0}\ge0$ and all $s\le t$, $s,t\in[t_{0},T]\setminus\mcN$. 

\no Fix now $t_{0}\ge0$ and  $(s,t)\in\D\left([t_{0},T]\setminus\mcN\right)$. For every partition $\D=\{s=t_{0}\le\dots\le t_{m}=t\}\subseteq[t_{0},T]\setminus\mcN$ we have 
\begin{align*}
 & \int [f^{2}(x,\xi,t)-|f|(x,\xi,t)]dxd\xi-\int [f^{2}(x,\xi,s)-|f|(x,\xi,s)]dxd\xi\\
= & \sum_{i=1}^{N}\int[ f^{2}(x,\xi,t_{i+1})-|f|(x,\xi,t_{i+1})]dxd\xi-\int [f^{2}(x,\xi,t_{i})-|f|(x,\xi,t_{i})]dxd\xi\\
\ge & -C\sum_{i=1}^{N}(t_{i+1}-t_{i})^{\a}\int_{t_{i}}^{t_{i+1}}\int m(x,\xi,r)dxd\xi dr\\
\ge & -\|\D\|^{\a}C\int_{s}^{t}\int m(x,\xi,r)dxd\xi dr.
\end{align*}
Letting $\|\D\|\to0$ yields, for all $t_{0}\ge0$ and $(s,t)\in\D\left([t_{0},T]\setminus\mcN\right)$,
\begin{align*}
\int [f^{2}(x,\xi,t)-|f|(x,\xi,t)]dxd\xi & \ge\int [f^{2}(x,\xi,s)-|f|(x,\xi,s)]dxd\xi,
\end{align*}
and, since $0\not\in\mcN$ and $\int (f^2(x,\xi,0) -|f|(x,\xi,0)) dx d\xi=0$,  for all $t\in[0,T]\setminus\mcN$ we conclude that \eqref{takis} holds. 
%this implies
%\[
%\int f^{(2)}(x,\xi,t)-|f|(x,\xi,t)dxd\xi=0\quad\forall t\in[0,T]\setminus\mcN,
%\]
%which by \eqref{eq:gen_kinetic_measure} implies $f=\chi(u)$ for some measurable $u:[0,T]\times\R^{N}\to\R$. Since,
%\begin{align*}
%\|u(t)\|_{L^{1}(\R^{N})} & =\int|\chi|(\xi,u(t,x))d\xi dx\\
% & =\int|f|(x,\xi,t)d\xi dx\\
% & =\|f(t)\|_{L^{1}(\R^{N}\times\R)}
%\end{align*}
%we have $u\in L^{\infty}([0,T];L^{1}(\R^{N}))$.
\end{proof}

\subsection*{Proof of Theorem~\ref{thm:existence}} The proof of Theorem \ref{thm:existence} now becomes a simple consequence of all the previous analysis. Indeed  Theorem \ref{thm:generalized_existence} yields a generalized pathwise rough entropy/kinetic  solution $f$ to \eqref{eq:scl} and Proposition \ref{prop:reconstruction} implies the existence of some  $u\in L^\infty([0,T];L^1(\R^N))$ such that 
  $f(x,\xi,t)=\chi(u(t,x),\xi).$ It then follows that  $u\in L^\infty([0,T];L^1(\R^N))$ is 
 a pathwise rough entropy/kinetic solution to \eqref{eq:scl}.

%\section{Existence of rough kinetic solutions}\label{sec:ex}

%\no We prove here the existence  of generalized pathwise entropy %rough kinetic 
%solutions to \eqref{eq:scl}, a fact which, In view of Proposition \ref{prop:reconstruction}, implies %the existence of a pathwise entropy solution.

\appendix

\section{Definitions and some estimates from rough paths theory}\label{app:RP}

\no We briefly recall the elements of rough paths theory used in this paper and for more details we refer to \cite{FV10}. As already mentioned in the introduction, we use the theory of rough paths to define a topology which is  stronger than the usual sup-norm on the space of driving signals. In the language of the rough path theory ordinary differential equations (ODE for short) are called controlled ODE and the signals are the controls.  It is not difficult to construct examples showing that controlled ODE do not, in general, depend continuously on the path of the signal and that is necessary to use higher order iterated integrals of the controls. 
\smallskip

\no This motivates the definition, for $x\in C^{1-\text{var}}([0,T];\R^N)$,  of the step $M$ signature
 $$S_M(x)_{0,T}=\left(1,\int_{0<u<T}dx_u,\dots,\int_{0<u_1 < \dots < u_M <T}dx_{u_1}\otimes \dots \otimes dx_{u_M}\right)$$
%for $x\in C^{1-\text{var}}([0,T];\R^N)$. 
We note that $S_M(x)$ takes values in the so-called truncated step-$M$ tensor algebra
 $$T^{M}(\R^{N})=\R\oplus\R^{N}\oplus(\R^{N}\otimes\R^{N})\oplus\ldots\oplus(\R^{N})^{\otimes M}.$$
%the so-called truncated step-$M$ tensor algebra. 
In fact, $S_M(x)$ takes values in the smaller set $G^{M}(\R^{N})\subset T^{M}(\R^{N})$ given by%where %$G^{M}(\R^{N})$ denotes the free step-$M$ nilpotent Lie group with $N$ generators. More precisely, for we define the step $M$-signature
 %
%Then $G^{M}(\R^{N})$ is given by
  $$G^{M}(\R^{N}):=\left\{ S_M(x)_{0,1} :\ x\in C^{1-\text{var}}([0,1];\R^N) \right\}.$$
On $G^{M}(\R^{N})$ we  introduce the Carnot-Caratheodory norm
  $$ \|g\| := \inf\left\{\ \int_0^1 |d\g|\ : \gamma \in C^{1-\text{var}}([0,1];\R^N) \text{ and } S_M(\g)_{0,1}=g \right\},$$
which gives rise to a so-called homogeneous metric on $G^{M}(\R^{N})$. Alternatively, for any $g\in T^{M}(\R^{N})$, we may set
  $$ |g|=|g|_{T^{M}(\R^{N})} = \max_{k=1\dots M}|\pi_k (g)|, $$
where $\pi_k$ is the projection of $g$ onto the $k$-th tensor level. This defines an inhomogeneous metric on $G^{M}(\R^{N})$. It turns out that the topologies induced by $\|\cdot\|$ and $|\cdot|$ are actually equivalent. 
\smallskip

\no For paths in $T^{M}(\R^{N})$ starting at the fixed point $e:=1+0+\ldots+0$ and $\beta\in(0,1]$,  one may then define $\beta$-Hölder metrics, extending the usual metrics for paths in $\R^{N}$ starting at zero.  The homogeneous $\beta$-Hölder metric is denoted by $\ensuremath{d_{\beta-\text{Höl}}}$ and the inhomogeneous one by $\ensuremath{\rho_{\beta-\text{Höl}}}$. A corresponding norm is defined by 
$\ensuremath{\|\cdot\|_{\beta-\text{Höl}}=d_{\beta-\text{Höl}}(\cdot,0)}$ where $0$ denotes the constant $e$-valued path.
\smallskip

\no A geometric $\beta$-Hölder rough path ${x}$ is a path in $T^{\lfloor1/\beta\rfloor}(\R^{N})$ which can be approximated by lifts of smooth paths in the $\ensuremath{d_{\beta-\text{Höl}}}$ metric. It can be shown that rough paths actually take values in  $G^{\lfloor1/\beta\rfloor}(\R^{N})$. We denote by 
%\begin{align*} 
  $C^{0,\beta-\text{H\"ol}}_0([0,T],G^{\lfloor 1/\beta \rfloor}(\R^N))$
%\end{align*} 
the space of geometric  $\beta$-Hölder rough paths.% where $\beta\in(0,1]$. 
\smallskip

\no We next recall some basic stability estimates for solutions to rough differential equations (RDE for short). For definiteness we consider RDE   of the form
\[
dx=V(x)\circ d{z},
\]
where ${z}$ is a geometric $\a$-Hölder rough path. It is well known (see, for example,  \cite{FV10}) that the RDE above has  a flow $\psi^{{z}}$ of solutions.
% $\psi^{{z}}$. 
The following is taken from \cite[Lemma 13]{CDFO13}.

\begin{lem}
\label{lem:hoelder_rp}Let $\a\in(0,1)$, $\g>\frac{1}{\a}\ge1$, $k\in\N$ and assume that 
$V  \in \Lip^{\g+k}(\R^{N};\R^{N}).$
For all $R>0$ there exist $C=C(R,\|V \|_{\Lip^{\g+k}})$
%\end{document}
 and $K=K(R,\|V\|_{\Lip^{\g+k}})$, which  are non-decreasing in all arguments,  such that,  
%\begin{align*}
%C & =C(R,\|V\|_{\Lip^{\g+k}})\\
%K & =K(R,\|V\|_{\Lip^{\g+k}})
%\end{align*}
for all geometric $\a$-Hölder rough paths ${z}^{1},{z}^{2}\in C_{0}^{\a-\Hoel}([0,T];G^{[\frac{1}{\a}]}(\R^{N}))$ with $\|{z}^{1}\|_{\a-\Hoel;[0,T]}$, $\|{z}^{2}\|_{\a-\Hoel;[0,T]}\le R$ and  all $n\in\{0,\dots,k\}$, %we have
\begin{align*}
\sup_{x\in\R^{N}}\|D^{n}(\psi^{{z}^{1}}-\psi^{{z}^{2}})(x)\|_{\a-\Hoel;[0,T]} & \le C\rho_{\a-\Hoel;[0,T]}({z}^{1},{z}^{2}),\\
\sup_{x\in\R^{N}}\|D^{n}((\psi^{{z}^{1}})^{-1}-(\psi^{{z}^{2}})^{-1})(x)\|_{\a-\Hoel;[0,T]} & \le C\rho_{\a-\Hoel;[0,T]}({z}^{1},{z}^{2})
\end{align*}
and, for all $n\in\{1,\dots,k\}$, 
\begin{align*}
\sup_{x\in\R^{N}}\|D^{n}\psi^{{z}^{1}}(x)\|_{\a-\Hoel;[0,T]}  \le K \  \text{ and} \
\sup_{x\in\R^{N}}\|D^{n}(\psi^{{z}^{1}})^{-1}(x)\|_{\a-\Hoel;[0,T]}  \le K.
\end{align*}
%for all $n\in\{1,\dots,k\}$. 
\end{lem}

\section{Convolution error estimates}

\no Now we present the proof of the error estimate.
\begin{lem}
\label{lem:err_est}Let $\vr_{t_{0},\ve}$ be as in \eqref{eq:eps_test_fct}. For each $R>0$ there exists $C=C(R)>0$ such that,\\
 for all $t\in [0,T]$ and all $\a$-Hölder rough paths ${z}$ with $\|{z}\|_{\a-\Hoel;[0,T]}\le R$,
\begin{align*}
\sup_{(x,\xi)\in \R^{N+1}}\int & \Big|\int\vr_{t_{0},\ve}(x',y,\xi',\eta,r)\partial_{\xi}\vr_{t_{0},\ve}(x,y,\xi,\eta,r)\\
 & +\partial_{\xi'}\vr_{t_{0},\ve}(x',y,\xi',\eta,r)\vr_{t_{0},\ve}(x,y,\xi,\eta,r)dyd\eta\Big|dx'd\xi'\le C(r-t_{0})^{\a}.
\end{align*}
%for some $C>0$ and all $r\in[t_{0},T]$. The constant $C=C(R)$ may be chosen uniformly for all $\a$-Hölder rough paths $z$ with $\|z\|_{\a-\Hoel;[0,T]}\le R$.
\end{lem}
\begin{proof}
In what follows for simplicity we write $X_{(x',\xi')}$ for $X_{(r,x',\xi')}(r-t_{0})$. We have:
\begin{align}
 & \vr_{t_{0},\ve}(x',y,\xi',\eta,r)\partial_{\xi}\vr_{t_{0},\ve}(x,y,\xi,\eta,r)+\partial_{\xi'}\vr_{t_{0},\ve}(x',y,\xi',\eta,r)\vr_{t_{0},\ve}(x,y,\xi,\eta,r)\nonumber \\[.5mm]
 & =\vr_{\ve}^{s}(X_{(x',\xi')}-y)\vr_{\ve}^{v}(\Xi_{(x',\xi')}-\eta)\partial_{\xi}[\vr_{\ve}^{s}(X_{(x,\xi)}-y)\vr_{\ve}^{v}(\Xi_{(x,\xi)}-\eta)]\nonumber \\[.5mm]
 & +\partial_{\xi'}[\vr_{\ve}^{s}(X_{(x',\xi')}-y)\vr_{\ve}^{v}(\Xi_{(x',\xi')}-\eta)]\vr_{\ve}^{s}(X_{(x,\xi)}-y)\vr_{\ve}^{v}(\Xi_{(x,\xi)}-\eta)\nonumber \\[.5mm]
 & =-\vr_{\ve}^{s}(X_{(x',\xi')}-y)\vr_{\ve}^{v}(\Xi_{(x',\xi')}-\eta)\partial_{\xi}X_{(x,\xi)}\partial_{y}\vr_{\ve}^{s}(X_{(x,\xi)}-y)\vr_{\ve}^{v}(\Xi_{(x,\xi)}-\eta)\label{eq:err_est_1-1}\\[.5mm]
 & -\partial_{\xi'}X_{(x',\xi')}\partial_{y}\vr_{\ve}^{s}(X_{(x',\xi')}-y)\vr_{\ve}^{v}(\Xi_{(x',\xi')}-\eta)\vr_{\ve}^{s}(X_{(x,\xi)}-y)\vr_{\ve}^{v}(\Xi_{(x,\xi)}-\eta)\nonumber \\[.5mm]
 & -\vr_{\ve}^{s}(X_{(x',\xi')}-y)\vr_{\ve}^{v}(\Xi_{(x',\xi')}-\eta)\vr_{\ve}^{s}(X_{(x,\xi)}-y)\partial _{\xi}\Xi_{(x,\xi)}\partial_{\eta}\vr_{\ve}^{v}(\Xi_{(x,\xi)}-\eta)\nonumber \\[.5mm]
 & -\vr_{\ve}^{s}(X_{(x',\xi')}-y)\partial_{\xi'}\Xi_{(x',\xi')}\partial_{\eta}\vr_{\ve}^{v}(\Xi_{(x',\xi')}-\eta)\vr_{\ve}^{s}(X_{(x,\xi)}-y)\vr_{\ve}^{v}(\Xi_{(x,\xi)}-\eta).\nonumber 
\end{align}
We next show the details  for the first two terms, since the second two terms are treated similarly. An integration by parts yields
\begin{align*}
 & -\int\vr_{\ve}^{s}(X_{(x',\xi')}-y)\vr_{\ve}^{v}(\Xi_{(x',\xi')}-\eta)\partial_\xi X_{(x,\xi)}\partial_{y}\vr_{\ve}^{s}(X_{(x,\xi)}-y)\vr_{\ve}^{v}(\Xi_{(x,\xi)}-\eta)dy\\
 & =\int\partial_{y}\vr_{\ve}^{s}(X_{(x',\xi')}-y)\vr_{\ve}^{v}(\Xi_{(x',\xi')}-\eta)\partial _\xi X_{(x,\xi)}\vr_{\ve}^{s}(X_{(x,\xi)}-y)\vr_{\ve}^{v}(\Xi_{(x,\xi)}-\eta)dy
\end{align*}
and thus
\begin{align*}
 & \int[-\vr_{\ve}^{s}(X_{(x',\xi')}-y)\vr_{\ve}^{v}(\Xi_{(x',\xi')}-\eta)\partial _{\xi}X_{(x,\xi)}\partial_{y}\vr_{\ve}^{s}(X_{(x,\xi)}-y)\vr_{\ve}^{v}(\Xi_{(x,\xi)}-\eta)\\
 & -\partial_{\xi'}X_{(x',\xi')}\partial_{y}\vr_{\ve}^{s}(X_{(x',\xi')}-y)\vr_{\ve}^{v}(\Xi_{(x',\xi')}-\eta)\vr_{\ve}^{s}(X_{(x,\xi)}-y)\vr_{\ve}^{v}(\Xi_{(x,\xi)}-\eta)]dyd\eta\\
= & \int\partial_{y}\vr_{\ve}^{s}(X_{(x',\xi')}-y)\vr_{\ve}^{s}(X_{(x,\xi)}-y)\vr_{\ve}^{v}(\Xi_{(x',\xi')}-\eta)\vr_{\ve}^{v}(\Xi_{(x,\xi)}-\eta)dyd\eta \ [\partial_{\xi}X_{(x,\xi)}-\partial_{\xi'}X_{(x',\xi')}]\\
% & \times [\partial_{\xi}X_{(x,\xi)}-\partial_{\xi'}X_{(x',\xi')}]\\
= & -\int[\vr_{\ve}^{s}(X_{(x',\xi')}-y)\partial_{y}\vr_{\ve}^{s}(X_{(x,\xi)}-y)\vr_{\ve}^{v}(\Xi_{(x',\xi')}-\eta)\vr_{\ve}^{v}(\Xi_{(x,\xi)}-\eta)]dyd\eta \ [\partial_{\xi}X_{(x,\xi)}-\partial_{\xi'}X_{(x',\xi')}].
% & \times [\partial_{\xi}X_{(x,\xi)}-\partial_{\xi'}X_{(x',\xi')}].
\end{align*}

\no Fix some $K>0$ and set 
$${\mathcal A}:=\{ x, x'  \in \R^{N},  \xi,\xi'  \in \R: |X_{(x',\xi')}-X_{(x,\xi)}|  \le K\ve \ \text{and}  \
|\Xi_{(x',\xi')}-\Xi_{(x,\xi)}|  \le K\ve \}.$$   % for some uniform $K>0$. 

\no Then 
\begin{align*}
 & |\int[-\vr_{\ve}^{s}(X_{(x',\xi')}-y)\vr_{\ve}^{v}(\Xi_{(x',\xi')}-\eta)\partial_{\xi}X_{(x,\xi)}\partial_{y}\vr_{\ve}^{s}(X_{(x,\xi)}-y)\vr_{\ve}^{v}(\Xi_{(x,\xi)}-\eta)\\[.5mm]
 & -\partial_{\xi'}X_{(x',\xi')}\partial_{y}\vr_{\ve}^{s}(X_{(x',\xi')}-y)\vr_{\ve}^{v}(\Xi_{(x',\xi')}-\eta)\vr_{\ve}^{s}(X_{(x,\xi)}-y)\vr_{\ve}^{v}(\Xi_{(x,\xi)}-\eta)] dy d\eta | 
 \\[.5mm]
& \le  \sup_{\mcA}
|\partial_{\xi}X_{(x,\xi)}-\partial_{\xi'}X_{(x',\xi')}|
 \int\vr_{\ve}^{s}(X_{(x',\xi')}-y)|\partial_{y}\vr_{\ve}^{s}(X_{(x,\xi)}-y)|\vr_{\ve}^{s}(\Xi_{(x',\xi')}-\eta)\vr_{\ve}^{s}(\Xi_{(x,\xi)}-\eta)dyd\eta .
\end{align*}
Next we note that 
\begin{align*}
 & \int\vr_{\ve}^{s}(X_{(x',\xi')}-y)|\partial_{y}\vr_{\ve}^{s}(X_{(x,\xi)}-y)|\vr_{\ve}^{s}(\Xi_{(x',\xi')}-\eta)\vr_{\ve}^{s}(\Xi_{(x,\xi)}-\eta)dyd\eta dx'd\xi'\\
 & =\int\vr_{\ve}^{s}(x'-y)|\partial_{y}\vr_{\ve}^{s}(X_{(x,\xi)}-y)|\vr_{\ve}^{s}(\xi'-\eta)\vr_{\ve}^{s}(\Xi_{(x,\xi)}-\eta)dyd\eta dx'd\xi'\\
 & =\int|\partial_{y}\vr_{\ve}^{s}(X_{(x,\xi)}-y)|\vr_{\ve}^{s}(\Xi_{(x,\xi)}-\eta)dyd\eta\\
 & =\int|\partial_{y}\vr_{\ve}^{s}(X_{(x',\xi')}-y)|dy  \le\frac{K}{\ve}.
\end{align*}
Thus
\begin{align}
 & \int |\int-\vr_{\ve}^{s}(X_{(x',\xi')}-y)\vr_{\ve}^{v}(\Xi_{(x',\xi')}-\eta)\partial _{\xi}X_{(x,\xi)}\partial_{y}\vr_{\ve}^{s}(X_{(x,\xi)}-y)\vr_{\ve}^{v}(\Xi_{(x,\xi)}-\eta)\nonumber \\
 & -\partial_{\xi'}X_{(x',\xi')}\partial_{y}\vr_{\ve}^{s}(X_{(x',\xi')}-y)\vr_{\ve}^{v}(\Xi_{(x',\xi')}-\eta)\vr_{\ve}^{s}(X_{(x,\xi)}-y)\vr_{\ve}^{v}(\Xi_{(x,\xi)}-\eta)dyd\eta |dx'd\xi'\label{eq:rough_est_1}\\
\le & \frac{K}{\ve}
%\sup_{\begin{array}{cc}
%x,\xi,x',\xi'\\
%|X_{(x',\xi')}-X_{(x,\xi)}| & \le K\ve\\
%|\Xi_{(x',\xi')}-\Xi_{(x,\xi)}| & \le K\ve
%\end{array}}
 \sup_{\mcA} |\partial_{\xi}X_{(x,\xi)}-\partial_{\xi'}X_{(x',\xi')}|.\nonumber 
\end{align}
From Lemma \ref{lem:hoelder_rp} we know that 
\[
C:=\sup_{x,\xi}\left\Vert D^{2}\left(\begin{array}{cc}
X_{(r,x,\xi)}(\cdot)\\
\Xi_{(r,x,\xi)}(\cdot)
\end{array}\right)\right\Vert _{\a-\Hoel;[0,r]}<\infty.
\]
Moreover, $D^{2}\left(\begin{array}{cc}
X_{(r,x,\xi)}(0)\\
\Xi_{(r,x,\xi)}(0)
\end{array}\right)=0$ for all $(r,x,\xi)$. Thus,
\begin{align*}
\sup_{x,\xi}\left\Vert D^{2}\left(\begin{array}{cc}
X_{(r,x,\xi)}(t)\\
\Xi_{(r,x,\xi)}(t)
\end{array}\right)\right\Vert  & =\sup_{x,\xi}\|D^{2}\left(\begin{array}{cc}
X_{(r,x,\xi)}(t)\\
\Xi_{(r,x,\xi)}(t)
\end{array}\right)-D^{2}\left(\begin{array}{cc}
X_{(r,x,\xi)}(0)\\
\Xi_{(r,x,\xi)}(0)
\end{array}\right)\|\\
 & \le t^{\a}\sup_{x,\xi}\|D^{2}\left(\begin{array}{cc}
X_{(r,x,\xi)}(\cdot)\\
\Xi_{(r,x,\xi)}(\cdot)
\end{array}\right)\|_{\a-\Hoel;[0,r]}\\
 & =Ct^{\a}.
\end{align*}
We conclude that there exists a uniform constant $C>0$ such that, for all $x,\xi,x', \xi'$,
\begin{align*}
 & |\partial_{\xi}X_{(r,x,\xi)}(r-t_{0})-\partial_{\xi'}X_{(r,x',\xi')}(r-t_{0})|\\
  & +|\partial_{\xi}\Xi_{(r,x,\xi)}(r-t_{0})-\partial_{\xi'}\Xi_{(r,x',\xi')}(r-t_{0})|\leq C (r-t_{0})^{\a}
\left\Vert 
\begin{array}{cc}
x-x'\\
\xi-\xi'
\end{array}\right\Vert .
%\quad\forall x,\xi,x',\xi'.
\end{align*}
Again by Lemma \ref{lem:hoelder_rp} we have that
\[
C:=\sup_{y,\eta}\left\Vert D\left(\begin{array}{cc}
Y_{(s,y,\eta)}(\cdot)\\
\z_{(s,y,\eta)}(\cdot)
\end{array}\right)\right\Vert _{\a-\Hoel;[s,t]}<\infty,
\]
and, since $\left(\begin{array}{cc}
Y_{(t_{0},y,\eta)}(r)\\
\z_{(t_{0},y,\eta)}(r)
\end{array}\right)$ is the inverse of $\left(\begin{array}{cc}
X_{(r,x,\xi)}(r-t_{0})\\
\Xi_{(r,x,\xi)}(r-t_{0})
\end{array}\right)$, it follows that, 
$$\text{ if} \  \left\Vert \begin{array}{cc}
X_{(x',\xi')}-X_{(x,\xi)}\\
\Xi_{(x',\xi')}-\Xi_{(x,\xi)}
\end{array}\right\Vert \le K\ve,  \ \text{then} \ 
\left\Vert \begin{array}{cc}
x-x'\\ 
\xi-\xi'
\end{array}\right\Vert \le C\ve.$$
%\]
%if $\left\Vert \begin{array}{cc}
%X_{(x',\xi')}-X_{(x,\xi)}\\
%\Xi_{(x',\xi')}-\Xi_{(x,\xi)}
%\end{array}\right\Vert \le K\ve.$ 
Hence,
%\begin{align*}
% & \sup_{\begin{array}{cc}
%x',\xi',x,\xi\\
%|X_{(x',\xi')}-X_{(x,\xi)}| & \le K\ve\\
%|\Xi_{(x',\xi')}-\Xi_{(x,\xi)}| & \le K\ve
%\end{array}}\left|\frac{d}{d\xi}X_{(r,x,\xi)}(r-t_{0})-\frac{d}{d\xi'}X_{(r,x',\xi')}(r-t_{0})\right|\\
% & +\left|\frac{d}{d\xi}\Xi_{(r,x,\xi)}(r-t_{0})-\frac{d}{d\xi'}\Xi_{(r,x',\xi')}(r-t_{0})\right|\le C(r-t_{0})^{\a}\ve,
%\end{align*}
$$ \sup_{\mathcal A}\left[  \left|\partial_{\xi}X_{(r,x,\xi)}(r-t_{0})-\partial_{\xi'}X_{(r,x',\xi')}(r-t_{0})\right| 
  +\left|\partial_{\xi}\Xi_{(r,x,\xi)}(r-t_{0})-\partial_{\xi'}\Xi_{(r,x',\xi')}(r-t_{0})\right|\right]\le C(r-t_{0})^{\a}\ve,   $$

\no which, in light of  \eqref{eq:rough_est_1}, finishes the proof. Note that the constant $C=C(R)$ may be chosen uniformly for all $\a$-Hölder rough paths ${z}$ with $\|{z}\|_{\a-\Hoel;[0,T]}\le R$.
\end{proof}
\smallskip

\bibliographystyle{amsalpha.bst}
\bibliography{refs}

% 
% \begin{comment}
% % expects file "refs.bib"
% 
% 
% % \bibliography{refs}
% \end{comment}

\end{document}